\documentclass[11pt,twoside]{amsart}

\numberwithin{equation}{section}
\usepackage{latexsym,amssymb,amsmath,amsthm,amsfonts}
\usepackage{graphicx}
\usepackage[table]{xcolor}
\usepackage{tikz}
\usetikzlibrary{matrix,arrows}
\definecolor{VerdeOlivo}{rgb}{0.3,0.5,0.1}
\definecolor{Magenta}{rgb}{.65,0.15,.2}
\newcommand{\bff}[1]{{\bf #1}}

\newtheorem{Theorem}{Theorem}[section] 
\newtheorem{Definition}[Theorem]{Definition} 
\newtheorem{Lemma}[Theorem]{Lemma} 
\newtheorem{Corollary}[Theorem]{Corollary} 
\newtheorem{Conjecture}[Theorem]{Conjecture} 
\newtheorem{Remark}[Theorem]{Remark} 
\newtheorem{Example}[Theorem]{Example} 
\newtheorem{Proposition}[Theorem]{Proposition}

\newtheorem*{Theorem36}{Theorem 3.6}
\newtheorem*{Theorem38}{Theorem 3.8}

\newtheorem*{Corollary39}{Corollary 3.9}
\definecolor{Gris}{gray}{0.3}

\begin{document}


\title[Critical ideals of trees]{Critical ideals of trees}


\author{Hugo Corrales}
\address{
Escuela Superior de Economia\\
Plan de Agua Prieta No. 66\\
Col. Plutarco El\'ias Calles\\
11340 Ciudad de M\'exico, M\'exico.
}
\email[H. ~Corrales]{hhcorrales@gmail.com}

\author{Carlos E. Valencia}
\address{
Departamento de
Matem\'aticas\\
Centro de Investigaci\'on y de Estudios Avanzados del
IPN\\
Apartado Postal
14--740 \\
07000 Ciudad de M\'exico, M\'exico.
} 
\email[C. ~Valencia\footnote{Corresponding author}]{cvalencia@math.cinvestav.edu.mx, cvalencia75@gmail.com}

\thanks{The authors were partially supported by SNI}




\begin{abstract}
Given a graph $G=(V, E)$, its generalized Laplacian matrix is given by 
\[
L(G,X_G)_{u,v}=
\begin{cases}
x_u&\text{if }u=v,\\
-m_{uv}&\text{if }u\neq v,
\end{cases}
\]
where $X_G=\{x_u\, | \, u\in V(G)\}$ is a set of indeterminates and $m_{uv}$ is the number of edges between $u$ and $v$.
The $j$-critical ideal of $G$ is the determinantal ideal generated by the minors of size $j$ of $L(G, X)$.
A $2$-matching of $G$ is a subset $\mathcal{M}$ of its edges such that every vertex of $G$ has at most two incident edges in $\mathcal{M}$. 
We give a combinatorial description of a set of generators of the $j$-critical ideal of a tree $T$ as a function of a set of special $2$-matchings, 
which we called minimal, of the graph $T^\ell$ obtained from $T$ by adding a loop at each of its vertices.
Also, we prove that the algebraic co-rank of $T$ is equal to the $2$-matching number of $T$, the maximum number of edges of a $2$-matching of $T$.
As a consequence, one can compute each invariant factor of the critical group of any graph $G$ such that $G\setminus v$ is a tree for some of its 
vertices $v$, as the greatest common divisor of the evaluation of some polynomials associated to the minimal $2$-matchings of $T^\ell$. 
For instance, in the regular case, we recover some of the results obtained by Levine~\cite{Levine} and Toumpakari~\cite{Toumpakary} about the critical group of a wired regular tree. 
Additionally, we can prove that the path $P_n$ is the unique simple graph with $n$ vertices and $n-1$ trivial critical ideals.
We conjecture that the set of generators that we give is a reduced Gr\"obner basis and we can prove this for the $|V(T)|-1$-critical ideal.
Finally, we apply the result in order to calculate the critical ideals of trees with depth two and some arithmetical trees associated to the reduction of elliptic curves of Kodaira type $I_n^*$.
\end{abstract}

\maketitle

\section{Introduction} 

The critical ideals of a graph were introduced in~\cite{Corrales} as a generalization of the critical group 
and the characteristic polynomial of the adjacency and Laplacian matrices of a graph. 
Critical ideals have been shown to be very useful.
For instance, in~\cite{Alfaro},  critical ideals were used to classify the simple graphs whose critical group has two of their invariant factors equal to one.
Also, in~\cite{AV} and~\cite{AVE}, they were  used to classify the simple graphs whose critical group 
has three of their invariant factors equal to one and bounded clique number.

Given a graph $G=(V,E)$, let $L(G,X_G)$ be the \emph{generalized Laplacian matrix} of $G$, which is given by
\[
L(G,X_G)_{u,v}=
\begin{cases}
x_u&\text{if }u=v,\\
-m_{uv}&\text{if }u\neq v,
\end{cases}
\]
where $X_G=\{x_u|u\in V(G)\}$ is the set of indeterminates indexed by the vertices of $G$ and $m_{uv}$ is the number of edges between $u$ and $v$.
For any $1\leq j\leq n$, the $j$-critical ideal of $G$ is given by
\[
I_j(G,X_G)=\big\langle\, j\textrm{-minors of } L(G,X_G)\big\rangle\subseteq \mathbb{Z}[X_G].
\]
The critical group of $G$, denoted by $K(G)$, is the torsion part of the cokernel of the matrix $L(G, \bff{d}_G)$ obtained from $L(G,X_G)$ 
by evaluating $X_G$ at the degree vector $\bff{d}_G$ of $G$.
Is not difficult to see that the critical group of $G$ can be obtained as an evaluation of a set of generators of its critical ideals, see~\cite[Proposition 3.6]{Corrales}.
Moreover, all the information of the critical group of $G$ is contained in its critical ideals.
For instance, if 
\[
\gamma(G)=\textrm{max}\{j\, | \, I_j(G,X)=\mathbb{Z}[X_G]\},
\] 
then $\textrm{rank}(K(G))\leq n-1-\gamma(G)$.
In contrast with the critical group of a graph, is easier to relate some combinatorial invariants of the graph to its critical ideals. 
For instance, \cite[Theorem 3.13]{Corrales} asserts that $\gamma(G)\leq\textrm{min}(2n-\alpha(G),2n-\omega(G)-1)$,
where $\alpha(G)$ and $\omega(G)$ are the stability and the clique numbers of $G$, respectively. 
In a similar way, the results obtained in~\cite{Alfaro} about the characterization of the connected graphs with $\gamma(G)\leq 2$ 
and with two invariant factors of the critical group of a graph suggest a most evident role of the combinatorial structure of $G$ for the critical ideals of a graph. 
Critical ideals are a very powerful tool to understood the critical group of a graph.
Moreover, the critical ideals of only one graph say something about the critical group of a big family of graphs.
For instance, the critical group of an arithmetical graph can be computed as an evaluation of the critical ideals of the base graph of the arithmetical graph.
Also, if $H$ is obtained from a graph $G$ by duplicating or replicating its vertices, 
then much information about the critical ideals and critical group of $H$ is contained in the critical ideals of $G$, see~\cite{twin}.

This paper focuses mainly on giving an explicit description of a set of generators of the critical ideals of a tree $T$, a connected graph without cycles. 
Our description of these critical ideals is based on the set of $2$-matchings of $T^\ell$,
the (non-simple) graph obtained from $T$ by adding a loop at each vertex of $T$. 
More precisely, a $2$-matching of a graph $G$ is a subset $\mathcal{M}$ of its edges 
such that every vertex of $G$ has at most two incident edges in $\mathcal{M}$.
We can think of a $2$-matching as a disjoint union of paths, where we are considering a loop as a path with only one vertex.
To each $2$-matching $\mathcal{M}$ of $T^\ell$ we associate a minor of $L(T,X_T)$, denoted by $d(\mathcal{M},X)$, 
in such a way that if $|\mathcal{M}|=j$, then $d(\mathcal{M},X)$ is a $j$-minor. 
This association leads to our main result.
\begin{Theorem36}
Let $1\leq j\leq n$, $T$ be a tree with $n$ vertices, and $\mathcal{V}_2^*(T^\ell,j)$ 
be the set of minimal $2$-matchings of $T^\ell$ of size $j$ (see Definition~\ref{defbas}).
Then 
\[
I_j(T,X_T)=\left\langle \,d(\mathcal{M},X)\,\Big| \, \mathcal{M}\in \mathcal{V}_2^*(T^\ell,j)\right\rangle.
\]
\end{Theorem36}
Moreover, we conjecture that $\{d(\mathcal{M},X)\,\Big| \, \mathcal{M}\in \mathcal{V}_2^*(T^\ell,j)\}$ 
is a reduced Gr\"obner basis for $I_j(T,X_T)$, see Conjecture~\ref{conj}.
We can prove this for the $n-1$-critical ideal, see Theorem~\ref{Groebner}.
In general, due to the complexity of the relations of the generators is very difficult to prove that 
$\{d(\mathcal{M},X)\,\Big| \, \mathcal{M}\in \mathcal{V}_2^*(T^\ell,j)\}$ is a reduced Gr\"obner basis for $I_j(T,X_T)$.

On the other hand, let $\gamma(G)=\textrm{max}\{j\,|\,I_j(G,X)=\mathbb{Z}[X_G]\}$ be the algebraic corank of $G$.
Is not difficult to see that if $G$ has $n$ vertices, then the critical group of $G$ has at most $n-1-\gamma(G)$ non trivial invariant factors.
A remarkable consequence of Theorem~\ref{TeoRed} is the characterization of the algebraic corank of a tree in terms of its combinatorics. 
If we set $\nu_2(G)$ as the maximum size of a $2$-matching on $G$, then we get the following result:
\begin{Theorem38}
If $T$ is a tree, then $\gamma(T)=\nu_2(T)$.
\end{Theorem38}

This result led us to prove Conjecture 4.12 given in~\cite{Corrales}.
\begin{Corollary39}
If $G$ is a simple graph with $n$ vertices, then $\gamma(G)=n-1$ if and only if $G=P_n$.
\end{Corollary39}

Although it seems that the set of trees is a restricted class of graphs, the calculations presented here 
can be applied for the calculation of the critical group of an important family of graphs, one of the largest so far.
In fact so far the family of trees is the largest family of graphs for which it has been able to calculate their critical ideals.

This article is organized as follows: In Section~\ref{matchings}, we introduce the concept of $2$-matching and present some of their basic properties, 
which be useful for establishing the main result of this paper.
In Section~\ref{trees} we establish the correspondence between $2$-matchings of $T^\ell$ and minors of $L(T,X_T)$ and illustrate it with several examples. 
After doing this we focus on the algebraic relations between the minors associated to $2$-matchings. 
In Section~\ref{grobner} we prove that the minors associated to the minimal $2$-matching of $T^\ell$ form a reduced Gr\"obner basis for the $n-1$ critical ideals of $T$.

Finally, Section~\ref{applications} is devoted to presenting three applications of the results obtained in the previous sections in the computation of the critical ideals and critical groups of trees.
Firstly we present some arithmetical trees associated to the reduction of elliptic curves of Kodaira type $I_n^*$. 
In the next subsection we study the critical ideals of the graph obtained from a regular tree by collapsing the leaves to a single vertex and 
obtain some results obtained by Levine~\cite{Levine} and Toumpakary~\cite{Toumpakary} about the critical groups of wired trees. 
Thirdly, we describe the critical ideals of all the trees with depth two.


\section{2-matchings of trees}\label{matchings}
In this section we introduce the concept of a $2$-matching of a graph, which plays an important role throughout this paper.
After that, we present some of its properties when the graph is a tree, which will be very useful for giving a description of its critical ideals.

\begin{Definition}
Let $G$ be a graph (possibly with loops and multiple edges) and $\mathcal{M}$ be a set of edges of $G$.
We say that $\mathcal{M}$ is a \emph{$2$-matching} if every vertex of $G$ has at most two incident edges in $\mathcal{M}$.
\end{Definition}
It is important to note that a loop $vv$ is counted twice as an incident edge of $v$.
The set of all $2$-matchings of a graph $G$ will be denoted by $\mathcal{V}_2(G)$.
Moreover, let $\mathcal{V}_2(G,j)$ be the set of $2$-matchings of $G$ of size $j$, that is, with $j$ edges.
Also, the \emph{$2$-matching number} of $G$, denoted by $\nu_2(G)$, is the maximum number of edges of a $2$-matching of $G$.
A \emph{maximum} $2$-matching of $G$ is a $2$-matching of $G$ of size $\nu_2(G)$.

The concept of a $2$-matching applies to any class of graphs, however in this chapter we are primarily interested in the case when $G$ is a tree.
If $T$ is a tree, then is not difficult to see that its $2$-matchings consist of a disjoint union of paths.
We recall that a loop is a path of length zero.
For instance, let $\mathcal{C}$ be the tree given in Figure~\ref{cat}.a.
If we take (see Figure~\ref{cat}.b and~\ref{cat}.c)
\[
\mathcal{M}_1=\{v_1v_2,v_2v_5,v_6v_7,v_6v_8\}\quad\textrm{and}\quad\mathcal{M}_2=\{v_1v_2,v_2v_3,v_2v_4,v_6v_8\},
\]
then $\mathcal{M}_1\in\mathcal{V}_2(\mathcal{C},4)$ and $\mathcal{M}_2\not\in\mathcal{V}_2(\mathcal{C})$ because $\mathcal{M}_2$ has $3$ incident edges on $v_2$.
\begin{figure}[h!]
\begin{center}
\begin{tabular}{c@{\extracolsep{10mm}}c@{\extracolsep{10mm}}c}
\begin{tikzpicture}
\draw {
	(0,0) node[draw, circle, fill=gray, inner sep=0pt, minimum width=4pt] (v5) {}
	(-1,0) node[draw, circle, fill=gray, inner sep=0pt, minimum width=4pt] (v2) {}
	(-2,0) node[draw, circle, fill=gray, inner sep=0pt, minimum width=4pt] (v1) {}
	(1,0) node[draw, circle, fill=gray, inner sep=0pt, minimum width=4pt] (v6) {}
	(2,0) node[draw, circle, fill=gray, inner sep=0pt, minimum width=4pt] (v9) {}
	(-1,1) node[draw, circle, fill=gray, inner sep=0pt, minimum width=4pt] (v3) {}
	(-1,-1) node[draw, circle, fill=gray, inner sep=0pt, minimum width=4pt] (v4) {}
	(1,1) node[draw, circle, fill=gray, inner sep=0pt, minimum width=4pt] (v7) {}
	(1,-1) node[draw, circle, fill=gray, inner sep=0pt, minimum width=4pt] (v8) {}
	
	(v1) to (v2) 
	(v2) to (v3) 
	(v2) to (v4) 
	(v2) to (v5) 
	(v5) to (v6) 
	(v6) to (v7) 
	(v6) to (v8) 
	(v6) to (v9)
	
	(v1)+(0.2,0.2) node {$v_1$}
	(v2)+(0.2,0.2) node {$v_2$}
	(v3)+(0.2,0.2) node {$v_3$}
	(v4)+(0.2,0.2) node {$v_4$}
	(v5)+(0.2,0.2) node {$v_5$}
	(v6)+(0.2,0.2) node {$v_6$}
	(v7)+(0.2,0.2) node {$v_7$}
	(v8)+(0.2,0.2) node {$v_8$}
	(v9)+(0.2,0.2) node {$v_9$}
};
\end{tikzpicture}
&
\begin{tikzpicture}
\draw {	
	(0,0) node[draw, circle, fill=gray, inner sep=0pt, minimum width=4pt] (v51) {}
	(-1,0) node[draw, circle, fill=gray, inner sep=0pt, minimum width=4pt] (v21) {}
	(-2,0) node[draw, circle, fill=gray, inner sep=0pt, minimum width=4pt] (v11) {}
	(1,0) node[draw, circle, fill=gray, inner sep=0pt, minimum width=4pt] (v61) {}
	(2,0) node[draw, circle, fill=gray, inner sep=0pt, minimum width=4pt] (v91) {}
	(-1,1) node[draw, circle, fill=gray, inner sep=0pt, minimum width=4pt] (v31) {}
	(-1,-1) node[draw, circle, fill=gray, inner sep=0pt, minimum width=4pt] (v41) {}
	(1,1) node[draw, circle, fill=gray, inner sep=0pt, minimum width=4pt] (v71) {}
	(1,-1) node[draw, circle, fill=gray, inner sep=0pt, minimum width=4pt] (v81) {}
	
	(v11) to (v21) 
	(v21) to (v51) 
	(v61) to (v71) 
	(v61) to (v81)
	(v21) edge[dashed] (v31) 
	(v21) edge[dashed] (v41) 
	(v51) edge[dashed] (v61) 
	(v61) edge[dashed] (v91)
	
	(v11)+(0.2,0.2) node {$v_1$}
	(v21)+(0.2,0.2) node {$v_2$}
	(v31)+(0.2,0.2) node {$v_3$}
	(v41)+(0.2,0.2) node {$v_4$}
	(v51)+(0.2,0.2) node {$v_5$}
	(v61)+(0.2,0.2) node {$v_6$}
	(v71)+(0.2,0.2) node {$v_7$}
	(v81)+(0.2,0.2) node {$v_8$}
	(v91)+(0.2,0.2) node {$v_9$}
};
\end{tikzpicture}
&
\begin{tikzpicture}
\draw {	
	(0,0) node[draw, circle, fill=gray, inner sep=0pt, minimum width=4pt] (v52) {}
	(-1,0) node[draw, circle, fill=gray, inner sep=0pt, minimum width=4pt] (v22) {}
	(-2,0) node[draw, circle, fill=gray, inner sep=0pt, minimum width=4pt] (v12) {}
	(1,0) node[draw, circle, fill=gray, inner sep=0pt, minimum width=4pt] (v62) {}
	(2,0) node[draw, circle, fill=gray, inner sep=0pt, minimum width=4pt] (v92) {}
	(-1,1) node[draw, circle, fill=gray, inner sep=0pt, minimum width=4pt] (v32) {}
	(-1,-1) node[draw, circle, fill=gray, inner sep=0pt, minimum width=4pt] (v42) {}
	(1,1) node[draw, circle, fill=gray, inner sep=0pt, minimum width=4pt] (v72) {}
	(1,-1) node[draw, circle, fill=gray, inner sep=0pt, minimum width=4pt] (v82) {}
	
	(v12) to (v22) 
	(v22) to (v32) 
	(v22) to (v42) 
	(v62) to (v82)
	(v22) edge[dashed] (v52) 
	(v52) edge[dashed] (v62) 
	(v62) edge[dashed] (v72) 
	(v62) edge[dashed] (v92)
	
	(v12)+(0.2,0.2) node {$v_1$}
	(v22)+(0.2,0.2) node {$v_2$}
	(v32)+(0.2,0.2) node {$v_3$}
	(v42)+(0.2,0.2) node {$v_4$}
	(v52)+(0.2,0.2) node {$v_5$}
	(v62)+(0.2,0.2) node {$v_6$}
	(v72)+(0.2,0.2) node {$v_7$}
	(v82)+(0.2,0.2) node {$v_8$}
	(v92)+(0.2,0.2) node {$v_9$}	
	};
\end{tikzpicture}\\
$(a)$ A caterpillar tree $\mathcal{C}$ & $(b)$ A $2$-matching of $\mathcal{C}$  & $(c)$ A non $2$-matching of $\mathcal{C}$
\end{tabular}
\caption{A caterpillar tree without two pairs of legs.}
\label{cat}
\end{center}
\end{figure}
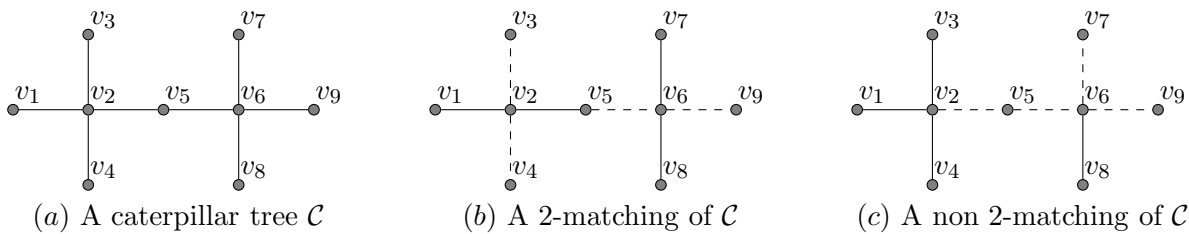

Now, we focus our attention on a special set of the $2$-matchings, the maximal ones.
\begin{Definition}\label{2matmax}
A $2$-matching $\mathcal{M}$ of a graph $G$ is called \emph{maximal} if there is no $2$-matching $\mathcal{N}$ such that $\mathcal{M}\subsetneq\mathcal{N}$. 
\end{Definition}

Note that a $2$-matching with size $\nu_2(G)$ is immediately maximal, but a maximal $2$-matching can have less than $\nu_2(G)$ edges. 
Even more, two maximal $2$-matchings can have different sizes.
\medskip

Is not difficult to check that the $2$-matching $\mathcal{M}_1$ (see Figure~\ref{cat}.b) is maximal.
The maximal $2$-matchings will play an important role in the description of the critical ideals of a tree.
In what follows, we will present the first property of the maximal $2$-matching of a tree.
Given two vertices $u,v$ of a tree $T$, let $P_{u,v}$ be the unique path in $T$ that joins $u$ and $v$.

\begin{Proposition}
If $\mathcal{M}$ is a maximal $2$-matching of a tree $T$, then there are $u\neq v$ leaves of $T$ such that $E(P_{u,v})\subseteq \mathcal{M}$.
\end{Proposition}
\begin{proof}
The proof follows by induction on the number of vertices of $T$. 
Is clear that the result is true for trees with three or less vertices.

Now, assume that the result is true for all the trees with $k$ or less vertices.
Let $T$ be a tree with $k+1$ vertices, $\mathcal{M}$ a maximal $2$-matching of $T$, $a$ a leaf of $T$, and $e=ab$ the edge of $T$ incident with $a$.
If $e\notin \mathcal{M}$, then $\mathcal{M}$ is a maximal $2$-matching of $T\setminus a$ and the result follows by the induction hypothesis.
On the other hand, if $e\in \mathcal{M}$, then $\mathcal{M}\setminus e$ is a maximal $2$-matching of $T\setminus a$.
Now, by the induction hypothesis, let there be $u\neq v$ leaves of $T\setminus a$ such that $E(P_{u,v})\subseteq \mathcal{M}\setminus e$.
If $b\neq u,v$, then the result follows. 
Otherwise, if $b=u$, then $a$ and $v$ are leaves of $T$ such that $E(P_{u,v})\subseteq \mathcal{M}$.
\end{proof}

In the following, we study the $2$-matching number of a tree when we delete one of its edges or vertices, 
in order to get a recursive way to calculate the set of its $2$-matchings and its $2$-matching number.
Before presenting the result, we introduce some concepts.
Given a tree $T$ and a vertex $v$, we say that $v$ is \emph{saturated} if any maximum $2$-matching of $T$ has two edges incident to it.
In a similar way, we say that an edge $e$ of $T$ is saturated when it appears in all the maximum $2$-matchings of $T$.

\begin{Lemma}\label{T-e}
If $T$ is a tree without loops and $e=uv$ is an edge of $T$, then
\[
\nu_2(T)-\nu_2(T\setminus e)=
\begin{cases}
0&\text{if and only if }u\text{ or }v\text{ are saturated in } T\setminus e,\\
1&\text{ if and only if } e \text{ is saturated in } T.
\end{cases}
\]
\end{Lemma}
\begin{proof}
Let $T_u$, respectively, $T_v$, be the connected components of $T\setminus e$ that contain the vertex $u$, respectively, $v$.
Let $\mathcal{M}$ be a maximum $2$-matching of $T$, $\mathcal{M}_u=\mathcal{M}\cap E(T_u)$, and $\mathcal{M}_v=\mathcal{M}\cap E(T_v)$.
Note that $\mathcal{M}_u$ and $\mathcal{M}_v$ are not necessarily maximum $2$-matchings of $T_u$ and $T_v$ respectively.
However, we can ensure that at least one of them is and the other is almost of maximum size ($\nu_2(T_v)-1$).
Since $|\mathcal{M}|=|\mathcal{M}_u|+|\mathcal{M}_u|+|\mathcal{M}\cap \{e\}|$ and $0\leq |\mathcal{M}\cap \{e\}|\leq 1$,
then $\nu_2(T)\leq \nu_2(T\setminus e)+1$.
In a similar way, taking maximum $2$-matchings of $T_u$ and $T_v$ we get that $\nu_2(T\setminus e)\leq \nu_2(T)$
and therefore $\nu_2(T)-1\leq \nu_2(T\setminus e) \leq \nu_2(T)$.

Now, $e$ is not saturated in $T$ if and only if there exist $\mathcal{N}_u$ and $\mathcal{N}_v$ maximum $2$-matchings of $T_u$ and $T_v$ respectively such that
$\mathcal{N}=\mathcal{N}_u\cup \mathcal{N}_v$ is a maximum $2$-matching of $T$.
This happens if and only if $\nu_2(T)= \nu_2(T\setminus e)$.
That is, $\nu_2(T)=\nu_2(T\setminus e)+1$ if and only if $e$ is saturated in $T$.

Finally, $e$ is saturated in $T$ if and only if each maximum $2$-matching $\mathcal{M}$ of $T$ is such that
$e\in \mathcal{M}$ and $\mathcal{M}\setminus e$ is a maximum $2$-matching of $T\setminus e$ if and only if 
${\rm deg}_{T[\mathcal{M}\setminus e]}(u),{\rm deg}_{T[\mathcal{M}\setminus e]}(v)\leq 1$.
This happens if and only if $u$ and $v$ are not saturated in $T\setminus e$.
That is, $e$ is saturated in $T$ if and only if $u$ and $v$ are not saturated in $T\setminus e$
or equivalently $\nu_2(T)=\nu_2(T\setminus e)$ if an only if $u$ or $v$ are saturated in $T\setminus e$.
\end{proof}

Now, we present how the $2$-matching number of a tree changes when we delete one if the trees vertices. 
In the following, $N_T(v)$ denotes the set of vertices of $T$ which are adjacent to $v$.

\begin{Lemma}\label{T-v}
Let $T$ be a tree without loops, $v$ a vertex of $T$, $N_T(v)=\{w_1,\ldots,w_s\}$, 
and $T_i$ the connected component of $T\setminus v$ that contains $w_i$.
Then
\[
\nu_2(T)-\nu_2(T\setminus v)=
\begin{cases}
2 & \text{if and only if } v \text{ is saturated in } T,\\
1 & \text{if and only if  there exists } 1\leq j \leq s \text{ such that } vw_j \text{ is saturated  and }\\
& w_i \text{ is saturated in } T_{i} \text{ for all } i\neq j,\\
0 & \text{if and only if } w_i \text{ is saturated in } T_{i} \text{ for all } 1\leq i \leq s.
\end{cases}
\]
\end{Lemma}
\begin{proof}
Given a maximum $2$-matching $\mathcal{M}$ of $T$, let $\mathcal{M}_i=\mathcal{M}\cap E(T_i)$.
Note that $\mathcal{M}_i$ is not necessarily a maximum $2$-matchings of $T_i$.
However, this is true in the following cases: $(i)$ $vw_i\notin \mathcal{M}$ and $(ii)$ $vw_i\in \mathcal{M}$ but $v$ is saturated in $T$.
Is clear that if $\mathcal{M}_i$ is not a maximum $2$-matching of $T_i$, then there exists a maximum $2$-matchings $\mathcal{M}'_i$ 
of $T_i$ such that $|\mathcal{M}'_i| >|\mathcal{M}_i|$.
{\bf Case $(i)$}. If $vw_i\notin \mathcal{M}$, then $\mathcal{M}'=(\mathcal{M}\setminus \mathcal{M}_i) \cup \mathcal{M}'_i$
is a $2$-matching of $T$ with $|\mathcal{M}'|> |\mathcal{M}|$, a contradiction.
{\bf Case $(ii)$}. If $vw_i\in \mathcal{M}$ and $v$ is saturated in $T$, then $\mathcal{M}'=[\mathcal{M}\setminus (\mathcal{M}_i\cup\{vw_i\})] \cup \mathcal{M}'_i$ 
is a $2$-matching of $T$ with $|\mathcal{M}'|\geq |\mathcal{M}|$ and ${\rm deg}_{\mathcal{M}'}(v)=1$, a contradiction.

On the other hand, since $T\setminus v=T_1\sqcup \cdots \sqcup T_s$, then $\nu_2(T\setminus v)=\nu_2(T_1)+\cdots+\nu_2(T_s)$ and
\[
2 \geq |\mathcal{M}\cap \delta_T(v)|=|\mathcal{M}\setminus \big(\bigcup_{i=1}^s \mathcal{M}_i \big)| \geq \nu_2(T)- \sum_{i=1}^s \nu_2(T_i)= 
\nu_2(T)-\nu_2(T\setminus v)=\nu_2(T)-\nu_2(T\setminus v),
\]
where $\delta_T(v)=\{vw \, | \, vw\in E(T)\}$.
That is, $\nu_2(T)-\nu_2(T\setminus v)\leq 2$.
Also, clearly $\nu_2(T\setminus v)\leq \nu_2(T)$ and therefore $\nu_2(T\setminus v) \leq \nu_2(T)\leq \nu_2(T\setminus v)+ 2$.

Now, if $\nu_2(T)-\nu_2(T\setminus v)=2$, then $|\mathcal{M}\cap \delta_T(v)|=2$ and $v$ is saturated in $T$.
Also, if $v$ is saturated in $T$, then the $\mathcal{M}_i$ are maximum $2$-matchings in $T_i$ and
\[
\nu_2(T)=|\mathcal{M}|=\sum_{i=1}^s |\mathcal{M}_i|+2=\sum_{i=1}^s \nu_2(T_i)+2=\nu_2(T\setminus v)+2.
\]
Also, it is not difficult to check that $\nu_2(T)=\nu_2(T\setminus v)$ if and only if $w_i$ is saturated in $T_{w_i}$ for all $w_i\in N_T(v)$.
Finally, if $\nu_2(T)=\nu_2(T\setminus v)+1$, then there exist $1\leq j\leq s$ such that $vw_j$ is saturated and $w_i$ is saturated in $T_{w_i}$ for all $w_i\in N_T(v)\setminus w_j$.
For the converse, we have the following:
\[
\nu_2(T\setminus v)= \sum_{i=1}^s \nu_2(T_i)\overset{w_i \text{ is saturated }\forall \, i\neq j}{=}\nu_2(T_v)+\nu_2(T_j)=\nu_2(T\setminus vw_j)\overset{{\rm Lemma}~\ref{T-e}}{=}\nu_2(T)-1,
\]
where $T_v$ is the connected component of $T\setminus vw_j$ that contains $v$.
\end{proof}

The next result proves that if $w_i$ is saturated in $T_i$, then $w_i$ is saturated in $T$.
Is not difficult to check that the converse is not true in general.

\begin{Proposition}\label{extension}
Let $T$ be a tree, $uv\in E(T)$, and $T_u$ the connected component of $T\setminus uv$ that contains $u$.
If $u$ is saturated in $T_u$, then $u$ is saturated in $T$.
\end{Proposition}
\begin{proof}
Let $T_v$ be the connected component of $T\setminus uv$ that contains $v$.
Since $u$ is saturated in $T_u$ and $T\setminus e=T_u\sqcup T_v$, then by Lemma~\ref{T-e}, $\nu_2(T)=\nu_2(T_u)+\nu_2(T_v)$.
Thus 
\[
\nu_2(T)-\nu_2(T \setminus u)=\nu_2(T_u)+\nu_2(T_v)-[\nu_2(T_u\setminus u)+\nu_2(T_v)]=\nu_2(T_u)-\nu_2(T_u\setminus u)\overset{u \text{ is saturated in } T_u}{=}2.
\]
Finally, by Lemma~\ref{T-v}, $u$ is saturated in $T$.
\end{proof}

As a consequence, we get the following result, which will be useful for proving one of the main results of this paper.
\begin{Corollary}\label{v+2}
If $T$ is a tree with at least three vertices, then it has at least one saturated vertex.
\end{Corollary}
\begin{proof}
This follows by Lemma~\ref{T-v} and Proposition~\ref{extension}.
The only tree that does not has a saturated vertex is the tree with only one edge.
\end{proof}

Now, we will study the $2$-matchings of the graphs obtained from a given graph by adding a loop at each of its vertices.


\subsection{Two matchings of $G^\ell$}
Given a simple graph $G$, let $G^\ell$ be the graph obtained from $G$ by adding a loop at each vertex of $G$.
That is, $E(G^\ell)=E(G)\cup\{uu\, | \, u\in V(G)\}$.
Given $\mathcal{M}\in \mathcal{V}_2(G^\ell)$, let $\ell(\mathcal{M})=\mathcal{M}\cap\{uu\, | \,u\in V(G)\}$.
Now, we introduce the concept of a minimal $2$-matching, which is central in the description of the critical ideals of a tree.

\begin{Definition}\label{defbas}
A $2$-matching $\mathcal{M}$ of $G^\ell$ is called minimal if there does not exist a $2$-matching $\mathcal{M}'$ 
of $G^\ell$ such that $\ell(\mathcal{M}')\subsetneq \ell(\mathcal{M})$ and $|\mathcal{M}|=|\mathcal{M}'|$.
The set of all minimal $2$-matchings of $G^\ell$ will be denoted by $\mathcal{V}_2^*(G^\ell)$. 
Moreover, let $\mathcal{V}_2^*(G^\ell,j)=\mathcal{V}_2^*(G^\ell)\cap \mathcal{V}_2(G^\ell,j)$ for any $1\leq j\leq n$.
\end{Definition}
\begin{Remark}
Note that the definition of a minimal $2$-matching makes sense only for graphs with a loop at each of their vertices. 
\end{Remark}

If $\mathcal{M}\in\mathcal{V}_2(G^\ell,j)\setminus \mathcal{V}_2(G,j)$ and $\mathcal{N}\in\mathcal{V}_2(G,j)$, 
then $|\mathcal{M}|=|\mathcal{N}|$ and $\ell(\mathcal{N})=\emptyset\subsetneq\ell(\mathcal{M})$.
Thus $\mathcal{V}_2^*(G^\ell,j)=\mathcal{V}_2(G,j)$ for all $1\leq j\leq \nu_2(G)$.
Moreover, next result shows that some maximal $2$-matchings of $T$ are part of a minimal $2$-matching of $T^\ell$.

\begin{Proposition}\label{minmax}
If $\mathcal{M}$ is a maximal $2$-matching of $T[N_T(\mathcal{M})]$, then 
\[
\mathcal{N}=\mathcal{M}\cup\{uu\, |\, u\notin V(\mathcal{M})\}
\]
is a minimal $2$-matching of $T^\ell$.
\end{Proposition}
\begin{proof} 
Assume that $\mathcal{N}$ is a not minimal $2$-matching of $T^\ell$.
Thus, there exists a $2$-matching $\mathcal{N}'$ of $T^\ell$ such that $|\mathcal{N}|=|\mathcal{N}'|$ and $\ell(\mathcal{N}')\subsetneq\ell(\mathcal{N})$. 
That is, $\mathcal{N}'$ has at least one more edge than $\mathcal{N}$. 
Since $\mathcal{M}$ is maximal on $N_T(\mathcal{M})$, $\mathcal{N}'$ must have an edge with at least one end in $V(T)\setminus V(\mathcal{M})$,
a contradiction to the fact that $\mathcal{N}$ has a loop at all the vertices of $V(T)\setminus V(\mathcal{M})$.
\end{proof}

\begin{Example}
Let $\mathcal{C}$ be the caterpillar tree considered in Figure~\ref{cat}.a.
It is not difficult to check that $\nu_2(\mathcal{C})=4$.
Thus, any minimal $2$-matching of $\mathcal{C}^\ell$ with at least one loop has at least size $5$.
The $2$-matching $\mathcal{M}_1=\{v_1v_2,v_2v_5,v_5v_6,v_6v_9, v_3v_3\}$ (see Figure~\ref{M5M6}.a)
is a minimal $2$-matching of $\mathcal{C}^\ell$ of size $5$ with only one loop.
Also, the $2$-matching given in Figure~\ref{M5M6}.b is a minimal $2$-matching of $\mathcal{C}^\ell$ of size $6$.

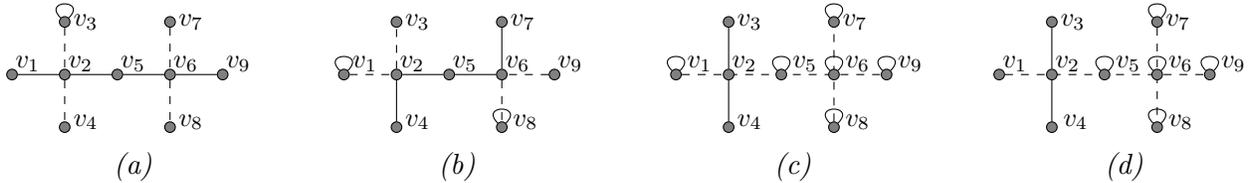
\begin{figure}[ht]
\begin{center}
\begin{tabular}{c@{\extracolsep{8mm}}c@{\extracolsep{8mm}}c@{\extracolsep{8mm}}c}
\begin{tikzpicture}[every loop/.style={}, scale=0.7]
\draw {
	(0,0) node[draw, circle, fill=gray, inner sep=0pt, minimum width=4pt] (v5) {}
	(-1,0) node[draw, circle, fill=gray, inner sep=0pt, minimum width=4pt] (v2) {}
	(-2,0) node[draw, circle, fill=gray, inner sep=0pt, minimum width=4pt] (v1) {}
	(1,0) node[draw, circle, fill=gray, inner sep=0pt, minimum width=4pt] (v6) {}
	(2,0) node[draw, circle, fill=gray, inner sep=0pt, minimum width=4pt] (v9) {}
	(-1,1) node[draw, circle, fill=gray, inner sep=0pt, minimum width=4pt] (v3) {}
	(-1,-1) node[draw, circle, fill=gray, inner sep=0pt, minimum width=4pt] (v4) {}
	(1,1) node[draw, circle, fill=gray, inner sep=0pt, minimum width=4pt] (v7) {}
	(1,-1) node[draw, circle, fill=gray, inner sep=0pt, minimum width=4pt] (v8) {}
	
	(v1) to (v2) 
	(v2) to (v5) 
	(v5) to (v6) 
	(v6) to (v9)
	(v3) edge[dashed] (v2)
	(v2) edge[dashed] (v4)  
	(v7) edge[dashed] (v6)
	(v6) edge[dashed] (v8)
	(v3) edge[loop] (v3)
	
	(v1)+(0.3,0.2) node {\small $v_1$}
	(v2)+(0.3,0.2) node {\small $v_2$}
	(v3)+(0.4,0) node { \small $v_3$}
	(v4)+(0.4,0.1) node {\small $v_4$}
	(v5)+(0.3,0.2) node {\small $v_5$}
	(v6)+(0.3,0.2) node {\small $v_6$}
	(v7)+(0.4,0) node {\small $v_7$}
	(v8)+(0.4,0.1) node {\small $v_8$}
	(v9)+(0.3,0.2) node {\small $v_9$}
	};
\end{tikzpicture}
&
\begin{tikzpicture}[every loop/.style={}, scale=0.7]
\draw {	
	(0,0) node[draw, circle, fill=gray, inner sep=0pt, minimum width=4pt] (v5) {}
	(-1,0) node[draw, circle, fill=gray, inner sep=0pt, minimum width=4pt] (v2) {}
	(-2,0) node[draw, circle, fill=gray, inner sep=0pt, minimum width=4pt] (v1) {}
	(1,0) node[draw, circle, fill=gray, inner sep=0pt, minimum width=4pt] (v6) {}
	(2,0) node[draw, circle, fill=gray, inner sep=0pt, minimum width=4pt] (v9) {}
	(-1,1) node[draw, circle, fill=gray, inner sep=0pt, minimum width=4pt] (v3) {}
	(-1,-1) node[draw, circle, fill=gray, inner sep=0pt, minimum width=4pt] (v4) {}
	(1,1) node[draw, circle, fill=gray, inner sep=0pt, minimum width=4pt] (v7) {}
	(1,-1) node[draw, circle, fill=gray, inner sep=0pt, minimum width=4pt] (v8) {}
	
	(v4) to (v2) 
	(v2) to (v5) 
	(v5) to (v6) 
	(v6) to (v7)
	(v1) edge[loop] (v1) 
	(v8) edge[loop] (v8)
	(v1) edge[dashed] (v2) 
	(v2) edge[dashed] (v3) 
	(v9) edge[dashed] (v6) 
	(v6) edge[dashed] (v8)
	
	(v1)+(0.4,0.2) node {\small $v_1$}
	(v2)+(0.3,0.2) node {\small $v_2$}
	(v3)+(0.4,0) node { \small $v_3$}
	(v4)+(0.4,0.1) node {\small $v_4$}
	(v5)+(0.3,0.2) node {\small $v_5$}
	(v6)+(0.3,0.2) node {\small $v_6$}
	(v7)+(0.4,0) node {\small $v_7$}
	(v8)+(0.45,0.1) node {\small $v_8$}
	(v9)+(0.3,0.2) node {\small $v_9$}
	};
\end{tikzpicture}
&
\begin{tikzpicture}[every loop/.style={}, scale=0.7]
\draw {
	(v1) edge[dashed] (v9) 
	(v8) edge[dashed] (v7) 	
	(0,0) node[draw, circle, fill=gray, inner sep=0pt, minimum width=4pt] (v5) {}
	(-1,0) node[draw, circle, fill=gray, inner sep=0pt, minimum width=4pt] (v2) {}
	(-2,0) node[draw, circle, fill=gray, inner sep=0pt, minimum width=4pt] (v1) {}
	(1,0) node[draw, circle, fill=gray, inner sep=0pt, minimum width=4pt] (v6) {}
	(2,0) node[draw, circle, fill=gray, inner sep=0pt, minimum width=4pt] (v9) {}
	(-1,1) node[draw, circle, fill=gray, inner sep=0pt, minimum width=4pt] (v3) {}
	(-1,-1) node[draw, circle, fill=gray, inner sep=0pt, minimum width=4pt] (v4) {}
	(1,1) node[draw, circle, fill=gray, inner sep=0pt, minimum width=4pt] (v7) {}
	(1,-1) node[draw, circle, fill=gray, inner sep=0pt, minimum width=4pt] (v8) {}
	
	(v2) to (v3) 
	(v2) to (v4)
	(v1) edge[loop] (v1) 
	(v5) edge[loop] (v5) 
	(v6) edge[loop] (v6) 
	(v7) edge[loop] (v7) 
	(v8) edge[loop] (v8) 
	(v9) edge[loop] (v9)
	
	(v1)+(0.45,0.2) node {\small $v_1$}
	(v2)+(0.3,0.2) node {\small $v_2$}
	(v3)+(0.4,0) node { \small $v_3$}
	(v4)+(0.4,0.1) node {\small $v_4$}
	(v5)+(0.45,0.2) node {\small $v_5$}
	(v6)+(0.45,0.2) node {\small $v_6$}
	(v7)+(0.4,0) node {\small $v_7$}
	(v8)+(0.45,0.1) node {\small $v_8$}
	(v9)+(0.45,0.2) node {\small $v_9$}
};
\end{tikzpicture}
&
\begin{tikzpicture}[every loop/.style={}, scale=0.7]
\draw {
	(v7) edge[dashed] (v8) 
	(v1) edge[dashed] (v9) 
	(0,0) node[draw, circle, fill=gray, inner sep=0pt, minimum width=4pt] (v5) {}
	(-1,0) node[draw, circle, fill=gray, inner sep=0pt, minimum width=4pt] (v2) {}
	(-2,0) node[draw, circle, fill=gray, inner sep=0pt, minimum width=4pt] (v1) {}
	(1,0) node[draw, circle, fill=gray, inner sep=0pt, minimum width=4pt] (v6) {}
	(2,0) node[draw, circle, fill=gray, inner sep=0pt, minimum width=4pt] (v9) {}
	(-1,1) node[draw, circle, fill=gray, inner sep=0pt, minimum width=4pt] (v3) {}
	(-1,-1) node[draw, circle, fill=gray, inner sep=0pt, minimum width=4pt] (v4) {}
	(1,1) node[draw, circle, fill=gray, inner sep=0pt, minimum width=4pt] (v7) {}
	(1,-1) node[draw, circle, fill=gray, inner sep=0pt, minimum width=4pt] (v8) {}
	
	(v2) to (v3) 
	(v2) to (v4)
	(v9) edge[loop] (v9) 
	(v6) edge[loop] (v6) 
	(v7) edge[loop] (v7) 
	(v8) edge[loop] (v8) 
	(v5) edge[loop] (v5)
	
	(v1)+(0.3,0.2) node {\small $v_1$}
	(v2)+(0.3,0.2) node {\small $v_2$}
	(v3)+(0.4,0) node { \small $v_3$}
	(v4)+(0.45,0.1) node {\small $v_4$}
	(v5)+(0.45,0.2) node {\small $v_5$}
	(v6)+(0.45,0.2) node {\small $v_6$}
	(v7)+(0.4,0) node {\small $v_7$}
	(v8)+(0.45,0.1) node {\small $v_8$}
	(v9)+(0.45,0.2) node {\small $v_9$}
	};
\end{tikzpicture}\\
(a)&(b)&(c)&(d)
\end{tabular}
\caption{$\mathcal{C}^{\ell}$ and some of its minimal $2$-matchings.}\label{M5M6}
\end{center}
\end{figure}
Let $\mathcal{M}_3=\{v_1v_1,v_3v_2,v_2v_4,v_5v_5, v_6v_6, v_7v_7,v_8v_8, v_9v_9\}$ be the $2$-matching given in Figure~\ref{M5M6}.c.
Using Proposition~\ref{minmax}, it is not difficult to check that $\mathcal{M}_3$ is a minimal $2$-matching of size $8$.
Moreover, $\mathcal{M}_4=\mathcal{M}_3 \setminus \{v_1v_1\}$ (see Figure~\ref{M5M6}.d) is also a minimal $2$-matching, of size $7$.
\end{Example}

Now, we give a recursive description of all the minimal $2$-matchings on $T^\ell$.

\begin{Proposition}
If $T$ is a tree and $e=uv\in E(T)$, then
\[
\mathcal{V}_2^*(T^\ell)\subseteq\{\mathcal{V}_2^*(T_u^\ell+e)^e\cup \mathcal{V}_2^*(T_v^\ell+e)^e\}\cup\{\mathcal{V}_2^*(T_u^\ell)\cup \mathcal{V}_2^*(T_v^\ell)\},
\]
where $T_x$ is the subtree of $T\setminus e$ that contains the vertex $x$ and $\mathcal{V}_2^*(G)^e$ is the set of minimal $2$-matchings of $G$ that contain the edge $e$.
\end{Proposition}
\begin{proof}
Let $\mathcal{M}$ be a minimal $2$-matching of $T^\ell$.
First, assume that $e\in\mathcal{M}$ and let $\mathcal{M}_u=\mathcal{M}\cap E(T_u^\ell+e)$ and $\mathcal{M}_v=\mathcal{M}\cap E(T_v^\ell+e)$.
As $\mathcal{M}=\mathcal{M}_u\cup \mathcal{M}_v$, it is enough to prove that $\mathcal{M}_u$ is a minimal $2$-matching of $T_u^\ell+e$.
Assume that $\mathcal{M}_u$ is not minimal. 
Thus there exists $\mathcal{M}'_u\in\mathcal{V}_2(T_u^\ell+e)$ such that 
$\ell(\mathcal{M}'_u)\subsetneq\ell(\mathcal{M}_u)$ and $|\mathcal{M}'_u|=|\mathcal{M}_u|$.
Note that $\ell(\mathcal{M}'_u\cup \mathcal{M}_v)\subsetneq \ell(\mathcal{M})$ and
\[
|\mathcal{M}'_u\cup \mathcal{M}_v|=\begin{cases}
|\mathcal{M}|+1&\mathrm{if\ }e\not\in \mathcal{M}'_u,\\
|\mathcal{M}|&\mathrm{if\ }e\in \mathcal{M}'_u.\end{cases}
\]
Thus, since $\mathcal{M}$ is minimal, $e\not\in \mathcal{M}'_x$ and $|\mathcal{M}'_u\cup\mathcal{M}_v|=|\mathcal{M}|+1$.
If we remove one loop (or an edge different from $e$) of $\mathcal{M}'_u$, then we get a $2$-matching $\mathcal{M}''_u$ of $T_u^\ell$
such that $|\mathcal{M}''_u\cup\mathcal{M}_v|=|\mathcal{M}|$ and $\ell(\mathcal{M}''_u\cup\mathcal{M}_v)\subsetneq\ell(\mathcal{M})$ 
which also contradicts the minimality of $\mathcal{M}$.
Thus, $\mathcal{M}_u$ is minimal on $T_u^\ell+e$.
As $e\in\mathcal{M}_u,\mathcal{M}_v$, $\mathcal{M}\in \mathcal{V}_2^*(T_u^\ell+e)^e\cup \mathcal{V}_2^*(T_v^\ell+e)^e$.

Finally, if we assume that $e\not\in \mathcal{M}$ and $\mathcal{M}_u=\mathcal{M}\cap T_u$ and $\mathcal{M}_v\cap T_v$, 
then the minimality of $\mathcal{M}_u$ and $\mathcal{M}_v$ can be deduced in a similar way.
\end{proof}


\section{Critical Ideals of Trees}\label{trees}
This section is devoted to establishing a relationship between the generators of the critical ideals of a tree $T$ and the $2$-matchings of $T^\ell$.
This relationship allows giving a complete and compact combinatorial description of the critical ideals of $T$.
Moreover, we prove that the critical ideals of $T$ are generated by the set of minimal $2$-matchings of $T^\ell$.

Since the $j$-critical ideal of a graph $G$ is generated by the $j$-minors of their generalized Laplacian matrix, 
then it only depends on the non-vanishing $j$-minors of $L(G,X_G)$.
Therefore, we begin by giving a description of the non-vanishing $j$-minors of the generalized Laplacian matrix of a tree.


\begin{subsection}{The non-vanishing minors of $L(T,X_T)$}
In this subsection we prove that the non-vanishing $j$-minors of $L(T,X_T)$ correspond to the $2$-matchings of $T$ of size $j$.
We begin by introducing some notation. 

Given a $2$-matching $\mathcal{M}$ of $T^\ell$, we associate to it the sets $t(\mathcal{M}),h(\mathcal{M})\subset V(T)$ as follows:
First, if $\mathcal{M}=\{v_{j_1}v_{j_2},v_{j_2}v_{j_3},\ldots,v_{j_m}v_{j_{m+1}}\}$ is a path, then 
\[
h(\mathcal{M})=\{v_{j_2},\ldots,v_{j_{m+1}}\}\text{ and } t(\mathcal{M})=\{v_{j_1},\ldots,v_{j_m}\}.
\]
That is, if $\overrightarrow{\mathcal{M}}$ is one of the two oriented paths obtained from $\mathcal{M}$,
then $h(\mathcal{M})$ are the heads and $t(\mathcal{M})$ the tails of their arcs.
Moreover, if $\mathcal{M}$ is non-connected and $\{\mathcal{M}_1,\ldots,\mathcal{M}_k\}$ are its connected components, 
then $h(\mathcal{M})=\cup_{s=1}^m h(\mathcal{M}_i)$ and $t(\mathcal{M})=\cup_{s=1}^m t(\mathcal{M}_i)$.

On the other hand, given a $2$-matching $\mathcal{M}$ of $T$, let $L(T,X)[t(\mathcal{M}),h(\mathcal{M})]$ be a square submatrix of $L(T,X)$. 
Clearly, $L(T,X)[t(\mathcal{M}),h(\mathcal{M})]$ has size $|\mathcal{M}|=\big|t(\mathcal{M})\big|=\big|h(\mathcal{M})\big|$. 
Let $a_\mathcal{M}$ be the leading coefficient of $\det\left(L(T,X)[t(\mathcal{M}),h(\mathcal{M})]\right)$ and 
\[
d(\mathcal{M},X)=\left\{\begin{array}{rl}
\det\left(L(T,X)[h(\mathcal{M}),t(\mathcal{M})]\right)&\textrm{if }a_\mathcal{M}>0,\\
-\det\left(L(T,X)[h(\mathcal{M}),t(\mathcal{M})]\right)&\textrm{if }a_\mathcal{M}<0.
\end{array}\right.
\]
Thus $d(\mathcal{M},X)$ is a generator of the $|\mathcal{M}|$-critical ideal of $T$.
As we will see in Lemma~\ref{Lemmarel}, $d(\mathcal{M},X)$ does not depend on the orientation of the paths on $\mathcal{M}$.
That is, the correspondence $\mathcal{M} \longmapsto  d(\mathcal{M},X)$ between $\mathcal{V}_2(T^\ell)$ and $\mathbb{Z}[X_G]$ is well defined. 
The next example will illustrate this correspondence between $2$-matchings and generators of the critical ideals of $T$.

\begin{Example}\label{Exam1}
Let $T$ be the tree given in Figure \ref{cat}. 
Is not difficult to see that 
\[
\mathcal{M}=\{v_1v_1,v_3v_2,v_2v_5,v_7v_6,v_6v_8,v_9v_9\}
\]
is a $2$-matching of $T^\ell$. 
Moreover, the paths $P_1=v_3v_2v_5$ and $P_2=v_7v_6v_8$ and the loops $L_1=v_1v_1$ and $L_2=v_9v_9$ are the connected components of $\mathcal{M}$.
Since $h(P_1)=\{v_2,v_5\}$, $t(P_1)=\{v_3,v_2\}$, $h(P_2)=\{v_6,v_8\}$, $t(P_2)=\{v_7,v_6\}$, $h(L_1)=\{v_1\}=t(L_1)$, and $h(L_2)=\{v_1\}=t(L_2)$, then
$h(\mathcal{M})=\{v_1,v_2,v_5,v_6,v_8,v_9\}$ and $t(\mathcal{M})=\{v_1,v_2,v_3,v_6,v_7,v_9\}$.
Thus 
\[
L(T,X)[h(\mathcal{M}),t(\mathcal{M})]=\left(\begin{array}{cccccc}
x_1&-1&0&0&0&0\\
-1&x_2&-1&0&0&0\\
0&-1&0&0&0&0\\
0&0&-1&x_6&-1&-1\\
0&0&0&-1&0&0\\
0&0&0&-1&0&x_9\\
\end{array}\right)
\]
and $d(\mathcal{M},X)=x_1x_9$.
\end{Example}

Moreover, as next lemma shows, $d(\mathcal{M},X)$ only depends on the loops of $\mathcal{M}$.

\begin{Lemma}\label{Lemmarel}
If $T$ is a tree and $\mathcal{M}$ a $2$-matching of $T^\ell$, then 
\[
d(\mathcal{M}, X)=\det(L(T[V(\ell(\mathcal{M}))],X)),
\]
where $T[V(\ell(\mathcal{M}))]$ is the subgraph of $T$ induced by the vertices of the loops in $\mathcal{M}$.
In particular, $d(\mathcal{M}, X)=\prod_{v\in V(\ell(\mathcal{M}))}x_v+\textit{``\itshape terms of lower degree''}$.
\end{Lemma}
\begin{proof}
First, is not difficult to prove that if $\mathcal{P}$ is a disjoint union of paths in a tree, then
\[
L(T,X)[h(\mathcal{P}),t(\mathcal{P})]\sim \left(\begin{array}{ccc}
1& * & * \\
0 & \ddots & * \\
0 & 0 & 1 \end{array}\right).
\]
Thus
\[
L(T,X)[h(\mathcal{M}),t(\mathcal{M})]\sim \left(\begin{array}{cccc}
L(T,X)[\ell(\mathcal{M}),\ell(\mathcal{M})]& * & * & * \\
0& 1& * & * \\
0& 0& \ddots & * \\
0& 0 & 0 & 1 \end{array}\right),
\]
where $\ell(\mathcal{M})=\{u_1,\ldots,u_r\}$ is the set of loops of $M$. Since
\[
L(T,X)[\ell(\mathcal{M}),\ell(\mathcal{M})]=\left(\begin{array}{ccc}
x_{u_1} & * & *\\
* & \ddots & * \\
* & * & x_{u_r}\end{array}\right)\]
and $\det\left(L(T,X)[h(\mathcal{M}),t(\mathcal{M})]\right)=\det\left(L(T,X)[l(\mathcal{M}),l(\mathcal{M})]\right)$, 
and now the result is clear.
\end{proof}

The next lemma is a partial converse of the previous result.

\begin{Lemma}\label{sobrematch}
If $f(X)$ is a non-vanishing minor of $L(T,X)$, 
then there exists an $\mathcal{M}\in\mathcal{V}_2(T^\ell)$ such that either $f(X)=d(\mathcal{M},X)$ or $f(X)=-d(\mathcal{M},X)$.
\end{Lemma}
\begin{proof}
Let $I,J\subseteq V(T)$ be such that $|I|=|J|\neq 0$ and $f(X)=\det\big(L(T,X)[I,J]\big)$.
Since $f(X)$ is non-zero, we can assume that all the entries in the main diagonal of $L(T,X)[I,J]$ are different from zero.
Now, let
\[
\mathcal{N}=\{v_{i_1}v_{j_1},\ldots,v_{i_t}v_{j_t}\},
\]
where $I=\{i_1,\ldots,i_t\}$ and $J=\{j_1,\ldots,j_t\}$.
Since $i_r\neq i_s$ and $j_r\neq j_s$ for all $r\neq s$, then $\mathcal{N}$ is a $2$-matching of $T^\ell$ with $|\mathcal{N}|\leq t=|I|$. 
If $|\mathcal{N}|< t$, then there exists $1\leq r<s\leq t$ such that $v_{i_r}v_{j_r}=v_{i_s}v_{j_s}$.
Since $i_r \neq i_s$, $i_r=j_s$, $j_r=i_s$ and $\mathcal{N}'=\mathcal{N} \cup \{v_{i_r}v_{j_s},v_{i_s}v_{j_r}\}\setminus \{v_{i_r}v_{j_r},v_{i_s}v_{j_s}\}$,
is a $2$-matching $\mathcal{N}'$ of $T^\ell$ with $|\mathcal{N}'|=|\mathcal{N}|+1$.
We can repeat this process until we get a $2$-matching $\mathcal{M}$ of size $t$ such that $I=t(\mathcal{M})$ and $J=h(\mathcal{M})$.
That is, $d(\mathcal{M}, X)=f(X)$.
\end{proof}

Lemma~\ref{sobrematch} gives us the following description of the critical ideals of $T$.
\begin{Corollary}\label{description1}
If $T$ is a tree, then 
\[
I_j(T, X)=\langle d(\mathcal{M}, X) \, | \, \mathcal{M}\in\mathcal{V}_2(T^\ell) \text{ with } |\mathcal{M}|=j\rangle.
\]
\end{Corollary}

Unfortunately, this description is not minimal.
For instance, is not difficult to find a tree $T$ and $2$-matchings $\mathcal{M}\neq \mathcal{N}$ of $T^\ell$ with $\ell(\mathcal{M})=\ell(\mathcal{N})$.
That is, $d(\mathcal{M}, X)=d(\mathcal{N}, X)$ (Lemma~\ref{Lemmarel}) and therefore the previous description of $I_j(T, X)$ contains repeated elements.
Moreover, the minors of $L(T,X)$ are related by very complex algebraic identities.
\end{subsection}

In what follows we exploit the combinatorial structure of $T$ in order to develop some identities between 
the minors of $L(T,X)$ which allows finding a better description of the critical ideals of a tree.
Before presenting the first of these identities, we fix some notation.
For any graph $G$, let $d(G,X)=\det(L(G,X))$.

\begin{Lemma}\label{TS}
If $T$ is a tree and $S\subseteq E(T)$, then
\[
d(T\setminus S,X)=\sum_{\mu\in \mathcal{V}_1(S)} d(T\setminus\!\! V(\mu),X),
\]
where $\mathcal{V}_1(S)$ is the set of matchings of the subgraph of $T$ induced by $S$.
\end{Lemma}
\begin{proof}
We use induction on $|S|$. 
First, let $S=\{uv\}$.
Since 
\[
\mathcal{V}_1(T)=\mathcal{V}_1(T\setminus uv)\bigcup\big\{\{uv\}\cup\mu\big|\mu\in\mathcal{V}_1(T\setminus \{u,v\})\big\}
\]
and $d(T,X)=\sum_{\mu\in\mathcal{V}_1(T)} (-1)^{|\mu|}\prod_{v\not\in V(\mu)} x_v,$ (\cite[Lemma 4.4]{Corrales}) 
we get the result, that is, $d(T\setminus uv,X)=d(T,X)+d(T\setminus\{u,v\},X)$.

Now, let $S=\{uv\}\cup S'$ with $|S'|>0$. 
If $T'=T\setminus S'$, then by the induction hypothesis
\[
d(T\setminus S,X)=d(T'\setminus uv,X)=d(T',X)+d(T'\setminus\{u,v\},X)
\]
and $d(T',X)=d(T\setminus S',X)=\sum_{\mu\in\mathcal{V}_1(S')} d(T\setminus V(\mu),X)$. 
On the other hand, since $T'\setminus\{u,v\}=T_{u,v}\setminus S''$, where $T_{u,v}=T\setminus\{u,v\}$ and $S''=\{e\in S' \, | \, u,v\not\in V(e)\}$, 
\[
d(T'\setminus\{u,v\},X)=d(T_{u,v}\setminus S'',X)=\sum_{\mu\in\mathcal{V}_1(S'')} d(T_{u,v}\setminus V(\mu),X).
\]
Moreover, since $\mathcal{V}_1(S)=\mathcal{V}_1(S')\bigcup\big\{\{uv\}\cup\mu\big|\mu\in\mathcal{V}(S'')\big\}$,
\[
d(T\setminus S,X)=\sum_{\mu\in\mathcal{V}_1(S')} d(T\setminus V(\mu),X)+\sum_{\mu\in\mathcal{V}_1(S'')} d(T_{u,v}\setminus V(\mu),X)
=\sum_{\mu\in \mathcal{V}_1(S)} d(T\setminus\!\! V(\mu),X).
\]
\end{proof}

The previous lemma is a fundamental result of this chapter, in fact almost all the identities 
between the generators of the critical ideals of a tree are derived from it.
For instance, we have the next corollary:

\begin{Corollary}\label{xminor}
If $\mathcal{M}$ is a $2$-matching of $T^\ell$ and $w$ a vertex such that $ww\notin \ell(\mathcal{M})$, then
\[
x_w\,d(\mathcal{M},X)= d(\mathcal{N},X) +\ \sum_{v\in U} d(\mathcal{M}\setminus\{vv\},X),
\]
where $\mathcal{N}=\{uv \, | \, uv \in \mathcal{M} \text{ and } u,v \neq w\} \cup\{ww\}$ and $U=\{v\in V(T)\,|\, vv\in \ell(\mathcal{M}),vw\in E(T)\}$.
\end{Corollary}
\begin{proof}
Let $T'=T[\ell(\mathcal{N})]$ and $S$ be the set of edges in $T'$ that contains $w$. 
Since $\mathcal{V}_1(S)=\emptyset \cup \{vw\, | \, vv\in \ell(\mathcal{M}), vw\in E(T)\}$, then applying Lemma~\ref{TS} to $T'$ and $S$ we get that
\[
d(T'\setminus S,X)=\sum_{\mu\in\mathcal{V}_1(S)}d(T'\setminus\! V(\mu),X)=d(T',X)\ +\ \sum_{v\in U} d(T'\setminus\!\!\{w,v\},X).
\]
Since $w$ and $T'\setminus S$ are not connected, then $d(T'\setminus S,X)=x_w d(T'\setminus w,X)=x_w d(\mathcal{M},X)$. 
On the other hand, by Lemma~\ref{Lemmarel}, $d(T',X)=d(\mathcal{N},X)$ and $d(T'\setminus\!\{w,v\},X)=d(\mathcal{M}\setminus\{vv\},X)$ for all $v\in U$. 
Combining these identities, we get the result.
\end{proof}

We will use these algebraic identities to find a compact description of the $j$-critical ideals of a 
tree and to prove that this description gives us a reduced Gr\"obner basis when $j=n-1$.
Corollary~\ref{xminor} allows us to prove one of the most important results of this article.

\begin{Theorem}\label{TeoRed}
If $T$ is a tree with $n$ vertices, then 
\[
I_j(T,X)=\left\langle \, \left\{d(\mathcal{M},X)\,| \, \mathcal{M}\in \mathcal{V}_2^*(T^\ell,j)\right\} \right\rangle \text{ for all }1\leq j\leq n.
\]
\end{Theorem}
\begin{proof}
By Lemma~\ref{sobrematch}, $I_j(T,X)\subseteq\{d(\mathcal{M},X)\big|\mathcal{M}\in\mathcal{V}_2(T^\ell,j)\}$.
Thus, we only need to prove that the minor of a non-minimal $2$-matching can be expressed in terms of minors associated to minimal $2$-matchings of the same size.

Let $\mathcal{M}$ be a non-minimal $2$-matching of size $j$.
Then, there is $\mathcal{N}\in \mathcal{V}_2(T,j)$ and $w\in V(T)$ such that $\ell(\mathcal{M})=\ell(\mathcal{N})\cup \{ww\}$.
Applying Proposition~\ref{xminor} to $\mathcal{N}$, we get that
\[
d(\mathcal{M},X)=x_w\,d(\mathcal{N},X)\ -\ \sum_{v\in U} d(\mathcal{N}\setminus \{vv\},X),
\]
where $U=\{v\in V(T)\,|\,vv\in\ell(\mathcal{N}), vw\in E(T)\}$.

For all $v\in U$, let $\mathcal{N}_{vw}=(\mathcal{N}\setminus \{vv\})\cup \{vw\}$.
Since $vw\not\in\mathcal{N}\setminus \{vv\}$, clearly $|\mathcal{N}_{vw}|=|\mathcal{N}\setminus \{vv\}|+1$ 
and therefore $\mathcal{N}_{vw}$ is a $2$-matching of $T^\ell$ of size $j$. 
On the other hand, since $\ell(\mathcal{N}_{vw})=\ell(\mathcal{N})\setminus\{vv\}$, by Lemma~\ref{Lemmarel} we get that $d(\mathcal{N}\setminus \{vv\},X)=d(\mathcal{N}_{uv},X)$. 
Therefore 
\[
d(\mathcal{M},X)=x_w d(\mathcal{N},X)\ -\ \sum_{v\in U} d(\mathcal{N}_{vw},X).
\]
Since $\ell(\mathcal{N})\subsetneq \ell(\mathcal{M})$ and $\ell(\mathcal{N})_{vw}\subsetneq \ell(\mathcal{M})$ for all $v\in U$,
we can repeat this process until we get an expression of $d(\mathcal{M},X)$ as an algebraic combination 
of the minors associated to some minimal $2$-matchings of $T^\ell$ of size $j$.
\end{proof}

The next example illustrates how Theorem~\ref{TeoRed} works.

\begin{Example}
Let $\mathcal{C}$ be the tree given in Figure~\ref{cat} and
\[
\mathcal{M}=\{v_1v_1,v_2v_2,v_3v_3,v_4v_4,v_5v_5,v_6v_6\}
\]
be a $2$-matching of $\mathcal{C}^\ell$ of size $6$. 
Since $\mathcal{M}_1=\{v_1v_1,v_2v_2,v_3v_3,v_4v_4,v_5v_5,v_6v_9\}$ is a $2$-matching 
of size $6$ with $\ell(\mathcal{M})=\ell(\mathcal{M}_1)\cup \{v_6v_6\}$, then $\mathcal{M}$ is non-minimal.
Thus $d(\mathcal{M},X)=x_6 d(\mathcal{M}_1,X)-d(\mathcal{M}_2,X)$, where $\mathcal{M}_2=\{v_1v_1,v_2v_2,v_3v_3,v_4v_4,v_5v_6,v_6v_9\}$.

In a similar way, since $\mathcal{M}_1$ and $\mathcal{M}_2$ are not minimal, then $d(\mathcal{M}_1,X)=x_5 d(\mathcal{M}_2,X)-d(\mathcal{M}_3,X)$ where $\mathcal{M}_3=\{v_1v_1,v_2v_5,v_3v_3,v_4v_4,v_5v_6,v_6v_9\}$ and $d(\mathcal{M}_2,X)=x_2 d(\mathcal{M}_3,X)-d(\mathcal{M}_4,X)-d(\mathcal{M}_5,X)-d(\mathcal{M}_6,X)$ with $\mathcal{M}_4=\{v_1v_2,v_3v_3,v_4v_4\} \cup\mathcal{P}$, $\mathcal{M}_5=\{v_1v_1,v_2v_3,v_4v_4\}\cup\mathcal{P}$, $\mathcal{M}_6=\{v_1v_1,v_2v_4,v_3v_3\}\cup\mathcal{P}$, and $\mathcal{P}=\{v_2v_5,v_5v_6,v_6v_9\}$.

Finally, since $\mathcal{M}_4,\mathcal{M}_5$ and $\mathcal{M}_6$ are minimal $2$-matchings and $d(\mathcal{M}_3,X)=x_1 d(\mathcal{M}_4,X)$, then
\begin{eqnarray*}
d(\mathcal{M},X)=(x_1\cdot p_{2,5,6}\!-\!p_{5,6})\cdot d(\mathcal{M}_4,X)\!-\!p_{5,6}\cdot d(\mathcal{M}_5,X)\!-\!p_{5,6}\cdot d(\mathcal{M}_6,X),
\end{eqnarray*}
where $p_{2,5,6}=x_2x_5x_6\!-\!x_2\!-\!x_6$ and $p_{5,6}=x_5x_6\!-\!1$.
In a similar way, we can get  that
\begin{eqnarray*}
d(\mathcal{M},X)&\!\!\!\!=\!\!\!\!&(x_4\cdot p_{2,5,6}\!-\!p_{5,6})\cdot d(\mathcal{M}_6,X)\!-\!p_{5,6}\cdot d(\mathcal{M}_4,X)\!-\!p_{5,6}\cdot d(\mathcal{M}_5,X),\\
&\!\!\!\!=\!\!\!\!&(x_3\cdot p_{2,5,6}\!-\!p_{5,6})\cdot d(\mathcal{M}_5,X)\!-\!p_{5,6}\cdot d(\mathcal{M}_4,X)\!-\!p_{5,6}\cdot d(\mathcal{M}_6,X),
\end{eqnarray*}
which gives us an expression of $d(\mathcal{M},X)$ in terms of minors associated to some minimal $2$-matchings of $\mathcal{C}^\ell$ of size $6$.
\end{Example}

The next result is a fundamental identity in the study of the critical ideals of trees,
which proves that the first $\nu_2(T)$ critical ideals of a tree are trivial.

\begin{Theorem}\label{TeoTreeGamma}
If $T$ is a tree, then $\gamma(T)=\nu_2(T)$.
\end{Theorem}
\begin{proof}
Let $\mathcal{M}$ be a maximum $2$-matching of $T$. By \ref{Lemmarel}, $d(\mathcal{M},X)=1$ and since $d(\mathcal{M},X)\in I_{\nu_2(T)}(T, X)$, then $I_{\nu_2(T)}(T, X)$ is trivial. Thus, we only need to prove that $I_{\nu_2(T)+1}(T)$ is non-trivial.
We will use induction on the number of vertices of the tree.
It is not difficult to check that the result is true for all the trees whose number of vertices is less than or equal to four,
therefore we can assume that $|V(T)|\geq 5$.
Let $k=\nu_2(T)+1$ and $v\in V(T)$.
By~\cite[Claim 3.12]{Corrales},
\[
I_k(T,X)\subseteq \left\langle x_v I_{k-1}(T\setminus v,X), I_{k-2}(T\setminus v,X),I_k(T\setminus v,X)\right\rangle.
\]
Moreover, since $I_{k}(T\setminus v,X)\subseteq I_{k-1}(T\setminus v,X)\subseteq I_{k-2}(T\setminus v,X)$, then
$I_k(T,X)\subseteq \left\langle x_v, I_{k-2}(T\setminus v,X)\right\rangle$.
 By the induction hypothesis, $\gamma(T\setminus v)=\nu_2(T\setminus v)$ for all $v\in V(T)$.
If we assume that $I_k(T,X)$ is trivial, then $I_{k-2}(T\setminus v,X)$ is trivial and therefore
\[
\nu_2(T)-1=k-2\leq\gamma(T\setminus v)=\nu_2(T\setminus v) \text{ for all } v\in V(T),
\]
which is a contradiction to Lemma~\ref{T-e}. 
\end{proof}

As a consequence, we get that $P_n$ is the only simple graph with $n$ vertices and  $\gamma(G)=n-1$.

\begin{Corollary}\label{Gamman1}
If $G$ is a simple graph with $n$ vertices, then $\gamma(G)=n-1$ if and only if $G=P_n$.
\end{Corollary}
\begin{proof}
$(\Rightarrow)$ If $G=P_n$, by Theorem \ref{TeoTreeGamma} $\gamma(G)=\nu_2(G)=\nu_2(P_n)=n-1$.
$(\Leftarrow)$ Let $G$ be a graph with $n$ vertices and $\gamma(G)=n-1$. 
Since $I_{n-1}(G,X)=\langle 1 \rangle$, by~\cite[Proposition 3.7]{Corrales} the critical group of $G$ must be trivial. 
Then by Kirchhoff's Matrix Tree Theorem~\cite[Theorem 6.2]{Biggs99}, $G$ is a tree.
By Theorem~\ref{TeoTreeGamma}, we get that $\nu_2(G)=n-1$. Thus, there exists a $2$-matching $\mathcal{M}$ of $T$ with size $n-1$. Let $P_{n_1},\ldots,P_{n_s}$ be paths on $G$ such that $\mathcal{M}=E(P_{n_1})\cup\cdots\cup E(P_{n_s})$. Since $G$ has $n$ vertices,
\[
n\geq |V(\mathcal{M})|=|V(P_{n_1})|+\cdots+|V(P_{n_s})|=|E(P_{n_1})|+\cdots+|E(P_{n_s})|+s=|\mathcal{M}|+s=n-1+s.
\]
Hence $s=1$ and so $n=|\mathcal{M}|+1=|E(P_{n_1})|+1=n_1$. Thus, since $G$ contains a path with $n$ vertices, $G=P_n$.
\end{proof}


\section{Gr\"obner basis of critical ideals}\label{grobner}

Usually the theory of Gr\"obner bases deals with ideals in a polynomial ring over a field.
However, in this section we deal with ideals in a polynomial ring over the integers.
There exists a theory of Gr\"obner basis over almost any kind of ring.
 
We recall some basic concepts of Gr\"obner bases. For more details, see~\cite{Adams}. 
First, let $\mathcal{P}$ be a principal ideal domain. A {\it monomial order} or {\it order term} in the polynomial ring $R=\mathcal{P}[x_1,\ldots,x_n]$ is a total order $\prec$ in the set of monomials of $R$ such that
\begin{description}
\item[(i)] $1\prec x^{\alpha}$  for all ${\bf 0} \neq {\bf \alpha} \in\mathbb{N}^n$, and  
\item[(ii)] if $x^{\alpha}\prec x^{\beta}$, then 
$x^{\alpha+\gamma}\prec x^{\beta+\gamma}$
for all $\gamma \in\mathbb{N}^n$,
\end{description}
where $x^{\alpha}=x_1^{\alpha_1} \cdots x_n^{\alpha_n}$.

Now, given a monomial order $\prec$ and $p\in R$, let ${\rm lt}(p)$, ${\rm lp}(p)$, and ${\rm lc}(p)$ 
be the {\it leading term}, the {\it leading power}, and the {\it leading coefficient} of $p$, respectively.
Given a subset $S$ of $R$, its leading term ideal is
\[
{\rm Lt}(S)=\langle {\rm lt}(s) \,|\, s\in S \rangle.
\]

A finite set of nonzero polynomials $B=\{b_1,\ldots, b_s\}$ of an ideal $I$ is called a {\it Gr\"obner basis} of $I$ 
with respect to an order term $\prec$ if ${\rm Lt}(B)={\rm Lt}(I)$.
Moreover, it is called {\it reduced} if ${\rm lc}(b_i)=1$ for all $1\leq i\leq s$ and no nonzero term in $b_i$
is divisible by any ${\rm lp}(b_j)$ for all $1\leq i\neq j\leq s$.

A good characterization of Gr\"obner bases is given in terms of the so called $S$-polynomials.
\begin{Definition}
Let $f$ and $f'$ be polynomials in $\mathcal{P}[X]$ and $B$ a set of polynomials in $\mathcal{P}[X]$. 
We say that $f$ {\it reduces strongly} to $f'$ modulo $B$ if
\begin{itemize}
\item ${\rm lt}(f')\prec {\rm lt}(f)$, and 
\item there exist $b\in B$ and $h\in \mathcal{P}[X]$ such that $f'=f-hb$. 
\end{itemize}
Moreover, if $f^*\in \mathcal{P}[X]$ can be obtained from $f$ in a finite number of reductions, we write $f\rightarrow_B f^*$.
\end{Definition}

That is, if $f=\sum_{j=1}^t p_{i_j}b_{i_j}+f^*$ with $p_{i_j} \in \mathcal{P}[X]$ and 
${\rm lt}(p_{i_j}b_{i_j})\neq {\rm lt}(p_{i_k}b_{i_k})$ for all $j\neq k$, then $f\rightarrow_B f^*$.

Now, given $f$ and $g$ polynomials in $\mathcal{P}[X]$, their {\it $S$-polynomial}, denoted by $S(f,g)$, is given by
\[
S(f,g)=\frac{c}{c_f}\frac{X}{X_f}\,f-\frac{c}{c_g}\frac{X}{X_g}\,g,
\]
where $X_f=lt(f)$, $c_f=lc(f)$, $X_g=lt(g)$, $c_g=lc(g)$, $X=\textrm{lcm}(X_f,X_g)$, and $c=\textrm{lcm}(c_f,c_g)$.

\medskip

The next lemma, known as Buchberger's criterion, gives us a useful criterion for checking whether a set of generators of an ideal is a  Gr\"obner basis.
\begin{Lemma}\label{Buch}
Let $I$ be an ideal of polynomials over a PID and $B$ be a generating set of $I$. 
Then $B$ is a Gr\"obner basis for $I$ if and only if $S(f,g)\rightarrow_B 0$  for all $f\neq g\in B$.
\end{Lemma}

In this paper we only work with the so called {\it degree lexicographic order}.

\begin{Definition}
Let $x^\alpha$ and $x^\beta$ be two monomials on $\mathcal{P}[x_1,\ldots,x_n]$,  then $x^\alpha\prec x^\beta$ whenever
\begin{itemize}
\item $\alpha_1+\cdots+\alpha_n<\beta_1+\cdots+\beta_n$,
\item or $\alpha_1+\cdots+\alpha_n=\beta_1+\cdots+\beta_n$ and exist $i=1,\ldots,n$ such that
\[
\alpha_1=\beta_1,\ \alpha_2=\beta_2,\ldots,\ \alpha_{i-1}=\beta_{i-1}\ \textrm{and}\ \alpha_i<\beta_i.
\]
\end{itemize}
\end{Definition}

In this section we prove that if $T$ is a tree with $n$ vertices, then $\{d(\mathcal{M},X)|\mathcal{M}\in\mathcal{V}_2^*(T^\ell,n-1)\}$ 
is not only a generating set, but also a Gr\"obner basis for $I_{n-1}(T)$. 
First we prove that for any $1\leq j\leq n$, a strong reduction by $\mathcal{V}_2(T^\ell,j)$ is equivalent to a strong reduction by $\mathcal{V}_2^*(T^\ell,j)$.

\begin{Proposition}\label{reductiontomin}
Let $1\leq j\leq n$ and $X_T=\{x_1,\ldots,x_s\}$.
If $f(x),g(x)\in\mathbb{Z}[X_T]$ are such that $f(x)\rightarrow_{\mathcal{V}_2(T^\ell,j)} g(x)$, then $f(x)\rightarrow_{\mathcal{V}_2^*(T^\ell,j)} g(x)$.
\end{Proposition}
\begin{proof}
Suppose that $d(\mathcal{M},X)\in\mathcal{V}_2(T^\ell,j)$ and $h(x)\in\mathbb{Z}[X_T]$ are such that $g(x)=f(x)-h(x)d(\mathcal{M},X)$ and $x_g\prec x_f$. 
If $\mathcal{M}$ is minimal, then there is nothing left to prove. 
On the other hand, if $\mathcal{M}$ is not minimal, then according to Theorem~\ref{TeoRed} there are 
$\mathcal{N}_1,\ldots,\mathcal{N}_s\in \mathcal{V}_2^*(T^\ell,j)$ and $t_1(x),\ldots,t_s(x)\in\mathbb{Z}[X_T]$ 
such that $d(\mathcal{M},X)=t_1(x)d(\mathcal{N}_1,X)+\cdots+t_s(x)d(\mathcal{N}_s,X)$. 
Thus
\[
g(x)=f(x)-\sum_{i=1}^{s} t_i(x)p(\mathcal{N}_i,X)h(x).
\]
Following the proof of Theorem~\ref{TeoRed}, we can ensure that for each $i=1,\ldots,s-1$
\[
\mathrm{lt}(t_i(x)p(\mathcal{N}_i,X))\prec\mathrm{lt}(t_{i+1}(x)p(\mathcal{N}_{i+1},X)).
\]
Thus, if 
\begin{eqnarray*}
f_1(x)&=&f(x)-t_1(x)p(\mathcal{N}_1,X)h(x),\\
f_2(x)&=&f_1(x)-t_2(x)p(\mathcal{N}_2,X)h(x),\\
&\vdots&\\ 
f_s(x)&=&f_{s-1}(x)-t_s(x)p(\mathcal{N}_s,X)h(x),
\end{eqnarray*}
then $x_{f_s}\prec \cdots \prec x_{f_1}\prec x_f$. 
Therefore $f(x)\rightarrow_{\mathcal{V}_2^*(T^\ell)} f_1(x) \rightarrow_{\mathcal{V}_2^*(T^\ell)} \cdots \rightarrow_{\mathcal{V}_2^*(T^\ell)} f_s(x)=g(x)$.
\end{proof}

Now, before proceeding to deal with the reduction of S-polynomials, we begin with the reduction of a monomial and a minor of size $n-1$. 
In what follow, if $e_1,e_2$ are two different edges in $T$, then $P(e_1,e_2)$ is the unique path in $T$ that joins $e_1$ and $e_2$.

\begin{Lemma}
\label{monDet}
If $T$ is a tree and $P$ is a non-empty path of $T$, then
\[
x_P\, d(T\setminus P,X)=d(T,X)+\sum_{e\in E(N_T(P))} d(T\setminus V(e),X)+\sum_{(e_1,e_2)\in\Lambda}\frac{x_{P(e_1,e_2)}}{x_{e_1}x_{e_2}}\,d(T\setminus P(e_1,e_2),X),
\]
where $\Lambda=\{(e_1,e_2)\in\mathcal{V}_1(N_T(P))|e_1,e_2\in E(N_T(P))\}.$
\end{Lemma}
\begin{proof}
Let $S=E(N_T(P))$. 
Clearly $V(P)$ is a free set of $T\setminus S$.
Thus, by Lemma~\ref{TS}
\[
x_P\, d(T\setminus P,X)=\!\!\!\!\!\!\sum_{\mu\in\mathcal{V}_1(S),|\mu|\leq 1}\!\!\!\!\!\! d(T\setminus V(\mu),X)\,\,\,\,\,+
\!\!\!\!\!\!\sum_{\mu\in\mathcal{V}_1(S),|\mu|\geq 2}\!\!\!\!\!\! d(T\setminus V(\mu),X).
\]
Each $\mu\in\mathcal{V}_1(S)$ with $|\mu|=2$ is a member of $\Lambda$. 
If $E_\mu$ is the neighborhood of $V(P(\mu))/V(\mu)$ in $T[V(P(\mu))/V(\mu)]$, then 
$\big\{\mu\in\mathcal{V}_1(N_T(P))\big| |\mu|\geq 2\big\}=\bigcup_{\mu\in\Lambda} \big\{\mu\cup\rho\big|\rho\in\mathcal{V}_1(E_\mu)\big\}$.
This relation allows us to write
\begin{eqnarray*}
\sum_{\mu\in\mathcal{V}_1(S),|\mu|\geq 2} d(T\setminus V(\mu),X)&=&
\sum_{\mu\in\Lambda}\ \sum_{\rho\in\mathcal{V}_1(E_\mu)} d(T\setminus V(\mu\cup\rho),X).
\end{eqnarray*}
For each $\mu\in\Lambda$, we apply Lemma~\ref{TS} to $T\setminus V(\mu)$ and $E_\mu$ to get that
\[
\sum_{\rho\in\mathcal{V}_1(E_\mu)} d(T\setminus V(\mu\cup\rho),X)=\frac{x_{P(\mu)}}{x_\mu}\ d(T\setminus P(\mu),X).
\]
\end{proof}

\begin{Remark}
Note that a $2$-matching $\mathcal{M}$ has size $n-1$ if and only if $T\setminus \ell(\mathcal{M})$ is a path (possibly of size zero). 
Thus, $T[\mathcal{M}]=T\setminus P$ for some path $P$ and $d(\mathcal{M},X)=d(T\setminus P,X)$. 
Conversely, for each path $P$, $T\setminus P=T[\mathcal{M}]$ for some $\mathcal{M}\in \mathcal{V}_2(T^\ell,n-1)$.
\end{Remark}

Now, we deal with the other case, of the product of a monomial and a minor of size $n-1$. 
Suppose that $P$ and $Q$ are non-empty paths of $T$ with $Q\subset P$. 
Then $P\setminus Q$ is composed of one or two paths, which we call $P_l$ and $P_r$ ($P_r$ could be empty). 
Let $L=E(N_{T\setminus Q}(P_l))$ and $R=E(N_{T\setminus Q}(P_r))$.

\begin{Proposition}\label{monsubDet}
Let $P$ be a path in a tree $T$ and $Q$ a non-empty subpath of $P$. 
If $L$ and $R$ are defined as above, then 
\[
\frac{x_P}{x_Q}\, d(T\setminus P,X)=d(T\setminus Q,X)+
\sum_{e\in L} \frac{x\,_{P(e,Q)}}{x_ex_Q}d(T\setminus P(e,Q),X)
\]
\[
+\sum_{e\in R} \frac{x\,_{P(Q,e)}}{x_Qx_e}d(T\setminus P(Q,e),X)+\sum_{e_l\in L}\sum_{e_r\in R}\frac{x_{P(e_l,e_r)}}{x_{e_l}x_Qx_{e_r}}\, d(T\setminus P(e_r,Q,e_l),X).
\]
\end{Proposition}
\begin{proof}
Set $T'=T\setminus Q$. As $L\cup R$ is the set of edges of $N_{T'}(V(P)\setminus V(Q))$ and $V(P)\setminus V(Q)$ is free in $T'\setminus S=T\setminus P$, by Lemma~\ref{TS}
\[
\frac{x_P}{x_Q}d(T\setminus P,X)=\sum_{\nu\in \mathcal{V}_1(L\cup R)} d(T'\setminus V(\nu),X).
\]
For each $e\in L$ let $P(e,Q)$ be the path in $T$ that join the vertices in $e$ and $Q$ and set $V_{e,Q}=V(P(e,Q))\setminus (V(e)\cup V(Q))$. 
If we set $S_{e,Q}=\{uv\in E(T)|u,v\in V_{e,Q}\}$, then $S_{e,Q}$ is a set of edges on $T_{e,Q}=T\setminus (V(e)\cup V(Q))$. 
Thus by Lemma~\ref{TS},
\[
d(T_{e,Q}\setminus S_{e,Q},X)=\sum_{\nu\in \mathcal{V}_1(S_{e,Q})} d(T_{e,Q}\setminus V(\nu),X).
\]
Since $T_{e,Q}\setminus S_{e,Q}=T\setminus P(e,Q)\cup V_{e,Q}$ and $\mathcal{V}_1(L)=\{\emptyset\}\cup_{e\in L} \{\{e\}\cup\mathcal{V}_1(S_{e,Q})\}$,
\[
\sum_{\nu\in\mathcal{V}_1(L)\setminus\{\emptyset\}}d(T'\setminus V(\nu),X)=\sum_{e\in L}\ \sum_{\nu\in\mathcal{V}_1(S_{e,Q})}d(T'\setminus V(\{e\}\cup\nu),X)\hspace*{2cm}
\]
\[
=\sum_{e\in L}\ \sum_{\nu\in\mathcal{V}_1(S_e)}d(T_{e,Q}\setminus V(\nu),X)=\sum_{e\in L}\ x_{V_{e,Q}}d(T\setminus P(e,Q),X).
\]

In the same way, we get an expression that involves $\mathcal{V}_1(R)$. 

Set $\mathcal{LR}$ as the (non-empty) matchings on $L\cup R$ that intersect both $L$ and $R$. 
For each $e_l\in L$ and $e_r\in R$, let $P(e_l,e_r)$ be the only path that joins $e_l$ and $e_r$. 
If we set $V_{e_l,e_r}=V(P(e_l,e_r))\setminus (V(e_l)\cup V(Q)\cup V(e_r))$ and $S_{e_l,e_r}=\{uv\in E(T)|u,v\in V_{e_l,e_r}\}$, 
then $S_{e_l,e_r}$ is a set of edges on $T_{e_l,e_r}=T\setminus (V(e_l)\cup Q \cup V(e_r))$. 
By Lemma~\ref{TS},
\[
d(T_{e_l,e_r}\setminus S_{e_l,e_r},X)=\sum_{\nu\in\mathcal{V}_1(S_{e_l,e_r})} d(T_{e_l,e_r}\setminus V(\nu),X).
\]
Noting that $T_{e_l,e_r}\setminus S_{e_l,e_r}= T\setminus P(e_l,e_r)+V_{e_r,e_l}$, $\mathcal{LR}=\bigcup_{e_l\in L}\bigcup_{e_r\in R} \{\{e_l,e_r\}\cup \mathcal{V}_1(S_{e_l,e_r})\}$ 
and that for each $\nu\in\mathcal{V}_1(S_{e_l,e_r})$, $T'\setminus V(\{e_l,e_r\}\cup\nu)=T_{e_l,e_r}\setminus V(\nu)$, we get
\[
\sum_{\nu\in\mathcal{LR}}d(T'\setminus V(\nu),X)=\sum_{e_l\in L}\sum_{e_r\in R}\ \sum_{\nu\in\mathcal{V}_1(S_{e_l,e_r})}d(T_{e_l,e_r}\setminus V(\nu),X)
\]
\[=\sum_{e\in L}\sum_{e_r\in R}\ d(T_{e_l,e_r}\setminus S_{e_l,e_r},X)=\sum_{e\in L}\sum_{e_r\in R}\ x_{V_{e_l,e_r}}d(T\setminus P(e,Q),X).
\]
This completes the proof of the theorem, as $\mathcal{V}_1(L\cup R)=\mathcal{V}_1(L)\cup\mathcal{V}_1(R)\cup \mathcal{LR}$.
\end{proof}

The two last results are used to establish the main result of this section.

\begin{Theorem}\label{Groebner}
If $T$ is a tree with $n$ vertices, then 
\[
\mathcal{B}_{n-1}=\{d(\mathcal{M},X)\, | \, \mathcal{M}\in \mathcal{V}_2^*(T^\ell,n-1)\}
\] 
is a reduced Gr\"obner basis for $I_{n-1}(T,X)$ with respect to the degree lexicographic order.
\end{Theorem}
\begin{proof}
By Proposition~\ref{Buch} and Theorem~\ref{TeoRed} we only need to prove that $S(f,g)\rightarrow_{\mathcal{B}_{n-1}} 0$ for all $f,g\in\mathcal{B}_{n-1}$.
If $\mathcal{M}_1,\mathcal{M}_2\in\mathcal{V}_2^*(T^\ell,n-1)$, then there are two paths $P_1$ and $P_2$ in $T$ such that $d(\mathcal{M}_i,X)=d(T\setminus P_i,X)$.

We can suppose that neither $P_1$ or $P_2$ is empty and that $P_1\neq P_2$, thus
\[
S(d(\mathcal{M}_1,X),d(\mathcal{M}_2,X))=x_{P_2^c\setminus P_1^c}\,d(T\setminus P_1,X)
-x_{P_1^c\setminus P_2^c}\,d(T\setminus P_2,X),
\]
where $P_i^c=V(T)\setminus V(P_i)$. If $P_1\cap P_2=\emptyset$, then
\[
S(d(\mathcal{M}_1,X),d(\mathcal{M}_2,X))=x_{P_1}\,d(T\setminus P_1,X)-x_{P_2}\,d(T\setminus P_2,X).
\]
By Lemma~\ref{monDet}, $S(d(\mathcal{M}_1,X),d(\mathcal{M}_2,X))\rightarrow_\mathcal{G} 0$.

If $P_1\cap P_2\neq\emptyset$, then this must be a path. If we set $Q=P_1\cap P_2$, then
\[
S(d(\mathcal{M}_1,X),d(\mathcal{M}_2,X))=\frac{x_{P_1}}{x_Q}\,d(T\setminus P_1,X)-\frac{x_{P_2}}{x_Q}\,d(T\setminus P_2,X),
\]
and by Proposition~\ref{monsubDet} $S(d(\mathcal{M}_1,X),d(\mathcal{M}_2,X))\rightarrow_\mathcal{G} 0$.
\end{proof}

The next result gives us an alternative and more compact description of the minimal $2$-matchings of $T^\ell$ of size $n-1$.
\begin{Proposition}\label{paths}
If $P_{u,v}$ is the unique path in $T$ that joins the vertices $u$ and $v$, then
\[
\mathcal{V}_2^*(T^\ell,n-1)=\{P_{u,v} \cup \{ww\, | \, w\notin V(P_{u,v})\}\,|\,u \text{ and }v\textrm{ are leaves of } T\}.
\]
\end{Proposition}
\begin{proof}
If $P$ is any path in $T$, then by Proposition~\ref{minmax}, $\mathcal{M}=P\cup\{ww\, |\, w\not\in V(P)\}$ is a minimal $2$-matching of size $n-1$ of $T^\ell$.
Therefore, we need to prove that if $\mathcal{M}\in\mathcal{V}_2^*(T^\ell,n-1)$, 
then $\mathcal{M}=P_{u,v}\cup \{ww\, | \, w\notin V(P_{u,v})\}$ for some $u$, $v$ leaves of $T$.

Let $\mathcal{M}\in\mathcal{V}_2^*(T^\ell,n-1)$.
If $\mathcal{M}$ has no edges, that is, $\mathcal{M}$ has $n-1$ loops, then suppose $u\in V(T)$ is such that $u\not\in V(\mathcal{M})$ and suppose $v\in V(T)$ satisfies  $uv\in E(G)$.
Since $\mathcal{M}'=\{uv\}\cup(\mathcal{M}\setminus \{vv\})$ has size $n-1$ and $\ell(\mathcal{M}')\subsetneq\ell(\mathcal{M})$, then $\mathcal{M}$ is not minimal.
Thus, $\mathcal{M}$ contains at least one path.
Furthermore, since $T$ has $n$ vertices, $\mathcal{M}$ has exactly one path. 
Let $P=\mathcal{M}\setminus\ell(\mathcal{M})$.
If $P'$ is a path in $T$ such that $P\subsetneq P'$, then $\mathcal{N}=P'\cup \{uu|u\not\in V(P')\}$ is a $2$-matching of size $n-1$ and $\ell(\mathcal{N})\subsetneq \ell(\mathcal{M})$, 
a contradiction to the minimality of $\mathcal{M}$.
Therefore, $P$ is maximal in the sense that $P$ is equal to $P_{u,v}$ for some leaves $u,v$ and $\mathcal{M}=P_{u,v}\cup\{ww\, |\, w\not\in V(P_{u,v})\}$.
\end{proof}

\begin{Remark}
If $T$ is a tree with $n$ vertices and $m$ leaves, then $\mathcal{B}_{n-1}$ contains $\binom{m}{2}$ polynomials.
\end{Remark}

Now, we present a conjecture about the minimality of the generating sets found in Theorem~\ref{TeoRed}.
In the next section we present several examples of the validity of this conjecture.

\begin{Conjecture}\label{conj}
If $T$ is a tree and $1\leq j\leq n$, then
\[
\mathcal{B}_j=\{d(\mathcal{M},X)\, |\, \mathcal{M}\in \mathcal{V}_2^*(T^\ell,j)\}
\]
is a reduced Gr\"obner basis for $I_j(T,X)$ with respect to the degree lexicographic order.
\end{Conjecture}

We finish this section with an example that shows how to get the $n-1$-critical ideal of a tree with $n$ vertices.

\begin{Example}
Let $n_1,n_2,n_3\geq 2$ and $J(n_1,n_2,n_3)$ be the tree with a vertex $r$ as root and three paths $P_{n_1}$, $P_{n_2}$, and $P_{n_3}$ from it of lengths $n_1,n_2$ and $n_3$, see Figure~\ref{JJJ}. 
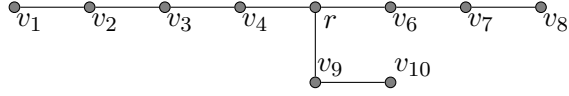
\begin{figure}[ht]
\begin{center}
\begin{tikzpicture}[every loop/.style={}]
\draw (0,0) {
	(-3.5,0) node[draw, circle, fill=gray, inner sep=0pt, minimum width=4pt] (v1) {}
	(-2.5,0) node[draw, circle, fill=gray, inner sep=0pt, minimum width=4pt] (v2) {}
	(-1.5,0) node[draw, circle, fill=gray, inner sep=0pt, minimum width=4pt] (v3) {}
	(-0.5,0) node[draw, circle, fill=gray, inner sep=0pt, minimum width=4pt] (v4) {}
	(0.5,0) node[draw, circle, fill=gray, inner sep=0pt, minimum width=4pt] (v5) {}
	(1.5,0) node[draw, circle, fill=gray, inner sep=0pt, minimum width=4pt] (v6) {}
	(2.5,0) node[draw, circle, fill=gray, inner sep=0pt, minimum width=4pt] (v7) {}
	(3.5,0) node[draw, circle, fill=gray, inner sep=0pt, minimum width=4pt] (v8) {}
	(0.5,-1) node[draw, circle, fill=gray, inner sep=0pt, minimum width=4pt] (v9) {}
	(1.5,-1) node[draw, circle, fill=gray, inner sep=0pt, minimum width=4pt] (v10) {}	
	(v1) to (v2) (v2) to (v3) (v3) to (v4) (v4) to (v5) (v5) to (v6) (v6) to (v7) (v7) to (v8)
	(v5) to (v9) (v9) to (v10)	
	(v1)+(0.2,-0.2) node {$v_1$}
	(v2)+(0.2,-0.2) node {$v_2$}
	(v3)+(0.2,-0.2) node {$v_3$}
	(v4)+(0.2,-0.2) node {$v_4$}
	(v5)+(0.2,-0.2) node {$r$}
	(v6)+(0.2,-0.2) node {$v_6$}
	(v7)+(0.2,-0.2) node {$v_7$}
	(v8)+(0.2,-0.2) node {$v_8$}
	(v9)+(0.2,0.2) node {$v_9$}
	(v10)+(0.3,0.2) node {$v_{10}$}
	};
\end{tikzpicture}
\caption{The tree $J(5,4,3)$.}
\label{JJJ}
\end{center}
\end{figure}

If $n=n_1+n_2+n_3-2=|J(n_1,n_2,n_3)|$, then by Theorem~\ref{Groebner}
\[
I_{n-1}(J(n_1,n_2,n_3),X)=\langle \det(P_{n_1}\!\setminus r,X),\det(P_{n_2}\!\setminus r,X),\det(P_{n_3}\!\setminus r,X)\rangle
\]
In particular $I_9(J(5,4,3),X)=\langle x_1x_2x_3x_4-x_1x_2-x_3x_4-x_1x_4+1,x_6x_7x_8-x_6-x_8,x_9x_{10}-1\rangle$.
\end{Example}


\section{Applications to the critical group}\label{applications}
Although the critical group of a tree is always trivial, the critical ideals of trees can be used to obtain 
information about the structure of the critical groups associated to a large class of interesting graphs. 
This section is devoted to presenting applications to trees of the results of Sections~\ref{trees} and~\ref{grobner}.


\subsection{Trees of depths one and two}
We begin with the trees of depth one, or stars, which are, along with the paths, the simplest trees. 
Note that the star with two leaves is denoted by $P_3$. 

\begin{Theorem}\label{stars}
Let $S(m)$ be the star with root $r$ and leaves $\{1,2,\ldots,m\}$. If $m\geq 3$, then $\gamma(S(m))=2$, for each $1\leq k\leq m-2$
\[
I_{2+k}(S(m),X)=\left\langle\ \prod_{s=1}^k x_{j_s} \, \Big| \, 1\leq j_1<\cdots <j_k\leq m\right\rangle,
\]
and $I_{m+1}=\langle x_rx_1\cdots x_m - x_1\cdots x_{m-1} - x_1\cdots x_{m-2}x_{m} - \cdots - x_2\cdots x_m \rangle$.
\end{Theorem}
\begin{proof}
It is pretty clear that $\nu_2(S(m))=\gamma(S(m))=2$ and since $|V(S(m))|=m+1$, there are $m-1$ non-trivial critical ideals of $S(m)$.
Moreover, is not difficult to see that for each $1\leq k\leq m-2$,
\begin{equation*}
\mathcal{V}_2^*(S(m)^\ell,k+2)=\left\{ \{j_sj_s\}_{s=1}^k \cup\{p_1r,r{p_2}\}\, \left|
\begin{array}{l}
1\leq p_1<p_2\leq m\text{ and for all} 1\leq s\leq k,\\
j_s\in \{1,\ldots,m\}\setminus\{p_1,p_2\}\end{array}\right.\right\}.
\end{equation*}
Thus, a straightforward application of Theorem~\ref{TeoRed} gives the result about $I_{2+k}(S(m),X)$. Finally, for $I_{m+1}(S(m),X)=\langle \det(S(m),X)\rangle$ we use Theorem~\ref{TS} with $S=E(S(m))$.
\end{proof}

Now, we continue, passing to the trees with depth two.
Let $s\geq 2$ and $T=T_2(m_1,\ldots,m_s)$ be the tree of depth two with $r$ as the root and $s$ branches with $m_i$ leaves each, see Figure~\ref{dep2}.
Note that $T_2(\emptyset)$ consists only of the root.
If $m_i\geq2$ for all $1\leq i\leq s$, then it is not difficult to see that $\nu_2(T)=2s$.
Since $n=|V(T)|=1+s+\sum_{i=1}^s m_i$, then $T$ has $n-2s=\sum_{i=1}^s m_i-s+1$ non-trivial critical ideals. 

In order to describe the non-trivial critical ideals, we need to characterize the minimal 2-matchings of $T^\ell$. 
Before doing this, we introduce some notation. Let $1,\ldots,s$ be the children of the root $r$ of $T$. 
For each $1\leq t\leq s$, let $t_1,\ldots,t_{m_t}$ be the children of $t$, let $S_t$ be the $t$ branch of $T$, that is, 
the star induced by the vertices $\{t, t_1,\ldots, t_{m_t}\}$  (see Figure~\ref{dep2}), and let $V_t$ denote a subset of $\{t_1,\ldots,t_{m_t}\}$. 
We use $P(u,v)$ to denote the edges of the path joining the vertices $u$ and $v$ in $T$.

\begin{Lemma}\label{T2mincar}
If $T=T_2(m_1,\ldots,m_s)$ and $\mathcal{M}\!\in\!\mathcal{V}_2^*(T^\ell,2s+l)$ with $l\geq 1$,
then
\[
\mathcal{M}=\begin{cases}
\begin{array}{l}
P({i_1}_{j_1},{i_2}_{j_2})\cup \displaystyle{\bigcup_{p=3}^{s} P({i_p}_{j_p},{i_p}_{k_p})}\cup V_{i_1}^\ell \cup\cdots\cup V_{i_s}^\ell
\end{array}
&\begin{array}{l}
1\leq j_p,k_p\leq m_{i_p}\textrm{ and } j_p\neq k_p,\\
\textrm{for each }1\leq p\leq s,
\end{array}
\\
\quad\\
\begin{array}{l}
P({i_1}_{j_1},r)\cup \displaystyle{\bigcup_{p=2}^{s} P({i_p}_{j_p},{i_p}_{k_p})}\cup V_{i_1}^\ell \cup\cdots\cup V_{i_s}^\ell
\end{array}
&\begin{array}{l}
1\leq j_p,k_p\leq m_{i_p}\textrm{ and } j_p\neq k_p,\\
\textrm{for each }1\leq p\leq s,
\end{array}
\\
\quad\\
\begin{array}{l}
\displaystyle{\bigcup_{p=1}^{s} P({i_p}_{j_p},{i_p}_{k_p})}\cup V_{i_1}^\ell \cup\cdots\cup V_{i_s}^\ell
\end{array}
&\begin{array}{l}
1\leq j_p,k_p\leq m_{i_p}\textrm{ and } j_p\neq k_p,\\
\textrm{for each }1\leq p\leq s,
\end{array}
\\
\quad\\
\begin{array}{l}
\displaystyle{P({i_1}_{j_1},{i_2}_{j_2})}\cup \displaystyle{\bigcup_{p=3}^{q} P({i_p}_{j_p},{i_p}_{k_p})}\cup
V_{i_1}^\ell \cup\cdots\cup V_{i_q}^\ell\\
\cup S_{i_{q+1}}^\ell\cup\cdots \cup S_{i_{s}}^\ell
\end{array}
&\begin{array}{l}
1\leq j_p,k_p\leq m_{i_p} \textrm{ and } j_p\neq k_p,\\
\textrm{for each }2\leq q<s\textrm{ and }1\leq p\leq q,
\end{array}
\\
\quad\\
\begin{array}{l}
\displaystyle{\bigcup_{p=1}^{q} P({i_p}_{j_p},{i_p}_{k_p})} \cup\{rr\}\cup V_{i_1}^\ell \cup\cdots\cup V_{i_q}^\ell\\
\cup S_{i_{q+1}}^\ell\cup\cdots \cup S_{i_{s}}^\ell
\end{array}
&\begin{array}{l}
1\leq j_p,k_p\leq m_{i_p}\textrm{ and } j_p\neq k_p,\\
\textrm{for each }0\leq q<s\textrm{ and }1\leq p\leq q.
\end{array}
\end{cases}
\]
\end{Lemma}
\begin{proof}
Let $I=V(\ell(\mathcal{M}))\cap \{1,\ldots,s\}$.
If $I=\emptyset$, then the minimality of $\mathcal{M}$ implies that the degree on each of the vertices $1,\ldots,s$ is $2$.
Thus, $E(\mathcal{M})$ is a maximum $2$-matching of $T$, so there are $1\leq j_p,k_p\leq m_{i_p}$ with $j_p\neq k_p$ such that
\[
E(\mathcal{M})=
P({i_1}_{j_1},{i_2}_{j_2})\cup \displaystyle{\bigcup_{p=3}^{s} P({i_p}_{j_p},{i_p}_{k_p})},\ 
P({i_1}_{j_1},r)\cup \displaystyle{\bigcup_{p=2}^{s} P({i_p}_{j_p},{i_p}_{k_p})}\textrm{ or }
\displaystyle{\bigcup_{p=1}^{s} P({i_p}_{j_p},{i_p}_{k_p})}.
\]

In the first two cases $\ell(\mathcal{M})\subseteq V_{1}^\ell \cup\cdots\cup V_{s}^\ell$ and in the third one $\ell(\mathcal{M})\subseteq \{rr\}\cup V_{1}^\ell \cup\cdots\cup V_{s}^\ell$.

If $I\neq\emptyset$, the minimality of $\mathcal{M}$ implies that $\mathcal{M}$ has degree degree $2$ in $r$. Thus, if $rr\not\in\ell(\mathcal{M})$, then there exist $i_1,i_2\in \{1,\ldots,n\}$ such that $i_1r,ri_2\in E(\mathcal{M})$, and since $\mathcal{M}$ is minimal, the degree of each of the vertices $1,\ldots,s$ must also be $2$. This ensures that there exists $2\leq q<s$ such that
\[
E(\mathcal{M})=P({i_1}_{j_1},{i_2}_{j_2})\cup \bigcup_{p=3}^{q} P({i_p}_{j_p},{i_p}_{k_p}),
\]

and thus $V(\ell(\mathcal{M}))=V_{i_1}^\ell \cup\cdots\cup V_{i_q}^\ell\cup S_{i_{q+1}}^\ell\cup\cdots \cup S_{i_{s}}^\ell$.
\medskip

Finally, if $rr\in\ell(\mathcal{M})$, similar arguments show that $E(\mathcal{M})=\bigcup_{p=1}^{q} P({i_p}_{j_p},{i_p}_{k_p})$ for some $0\leq q<s$. Thus, $\ell(\mathcal{M})=V_{i_1}^\ell \cup\cdots\cup V_{i_q}^\ell\cup S_{i_{q+1}}^\ell\cup\cdots \cup S_{i_{s}}^\ell$.
\end{proof}

Lemma~\ref{T2mincar} gives us a complete description of the critical ideals of $T=T_2(m_1,\ldots,m_s)$.
For example, if $\mathcal{M}\in\mathcal{V}_2^{*}(T,2s+1)$, then $|E(\mathcal{M})|=2s$ and $ \ell(\mathcal{M})\in V_i^\ell$ for some $i$, $1\leq i\leq s$. Thus,
\[
I_{2s+1}(T,X)=\langle x_v\,|\,v\textrm{ is a leaf of }T\rangle
\]

In what follows, we give a description of the critical ideals of some trees of depth two with tree branches.
\begin{figure}[ht]
\begin{center}
\begin{tikzpicture}
\draw (0,0) {
	(0,0) node[draw, circle, fill=gray, inner sep=0pt, minimum width=4pt] (r) {}
	
	(-3.00,-1) node[draw, circle, fill=gray, inner sep=0pt, minimum width=4pt] (v1) {}
	(-0.00,-1) node[draw, circle, fill=gray, inner sep=0pt, minimum width=4pt] (v2) {}
	( 3.00,-1) node[draw, circle, fill=gray, inner sep=0pt, minimum width=4pt] (v3) {}
	
	(-3.75,-2) node[draw, circle, fill=gray, inner sep=0pt, minimum width=4pt] (v11) {}
	(-3.25,-2) node[draw, circle, fill=gray, inner sep=0pt, minimum width=4pt] (v12) {}
	(-3.00,-2) node[draw, circle, fill=gray, inner sep=0pt, minimum width=1pt] () {}
	(-2.75,-2) node[draw, circle, fill=gray, inner sep=0pt, minimum width=1pt] () {}
	(-2.50,-2) node[draw, circle, fill=gray, inner sep=0pt, minimum width=1pt] () {}
	(-2.25,-2) node[draw, circle, fill=gray, inner sep=0pt, minimum width=4pt] (v1m) {}
	
	(-0.75,-2) node[draw, circle, fill=gray, inner sep=0pt, minimum width=4pt] (v21) {}
	(-0.25,-2) node[draw, circle, fill=gray, inner sep=0pt, minimum width=4pt] (v22) {}
	(0.00,-2) node[draw, circle, fill=gray, inner sep=0pt, minimum width=1pt] () {}
	(0.25,-2) node[draw, circle, fill=gray, inner sep=0pt, minimum width=1pt] () {}
	( 0.50,-2) node[draw, circle, fill=gray, inner sep=0pt, minimum width=1pt] () {}
	( 0.75,-2) node[draw, circle, fill=gray, inner sep=0pt, minimum width=4pt] (v2m) {}
	
	( 2.25,-2) node[draw, circle, fill=gray, inner sep=0pt, minimum width=4pt] (v31) {}
	( 2.75,-2) node[draw, circle, fill=gray, inner sep=0pt, minimum width=4pt] (v32) {}
	( 3.00,-2) node[draw, circle, fill=gray, inner sep=0pt, minimum width=1pt] () {}
	( 3.25,-2) node[draw, circle, fill=gray, inner sep=0pt, minimum width=1pt] () {}
	( 3.50,-2) node[draw, circle, fill=gray, inner sep=0pt, minimum width=1pt] () {}
	( 3.75,-2) node[draw, circle, fill=gray, inner sep=0pt, minimum width=4pt] (v3m) {}
	
	(r) to (v1) (r) to (v2) (r) to (v3)
	(v1) to (v11) (v1) to (v12) (v1) to (v1m)
	(v2) to (v21) (v2) to (v22) (v2) to (v2m)
	(v3) to (v31) (v3) to (v32) (v3) to (v3m)
	
	(r)+(0,0.2) node {$r$}
	(v1)+(-0.2,0.2) node {\small $1$}
	(v2)+(-0.2,0.2) node {\small $2$}
	(v3)+(0.2,0.2) node {\small $3$}
	(v11)+(-0.2,-0.3) node {\footnotesize $1_1$}
	(v12)+(0.0,-0.3) node {\footnotesize $1_2$}
	(v1m)+(0.1,-0.3) node {\footnotesize $1_{m_1}$}
	(v21)+(-0.2,-0.3) node {\footnotesize $2_1$}
	(v22)+(0.0,-0.3) node {\footnotesize $2_2$}
	(v2m)+(0.1,-0.3) node {\footnotesize $2_{m_2}$}
	(v31)+(-0.2,-0.3) node {\footnotesize $3_1$}
	(v32)+(0.0,-0.3) node {\footnotesize $3_2$}
	(v3m)+(0.1,-0.3) node {\footnotesize $3_{m_3}$}
		};
\end{tikzpicture}
\caption{The tree $T_2(m_1,m_2,m_3)$.}\label{dep2}
\end{center}
\end{figure}
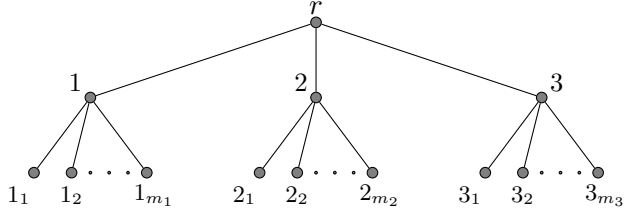

According to Proposition~\ref{T2mincar}, the critical ideals of $T_2(m_1,m_2,m_3)$ have two types of generators: 
monomials and products of a monomial with the determinant of a tree of depth one. 
Thus, for each $I\subseteq \{1,2,3\}$, let $Q_I=\det(L(\sqcup_{i\in I}S_i,X))$. 
Also, let 
\[
P_{r,s,t}^i\!=\!\left\{ \prod_{l=1}^{r_1} x_{1_{i_l}}\!\cdot\! \prod_{l=1}^{s_1} x_{2_{j_l}}\!\cdot\! \prod_{l=1}^{t_1} x_{3_{k_l}} \,\left|
\begin{array}{l}
1\!\leq\! i_1\! <\!\cdots \!<\! i_{r_1}\!\leq\! m_1,1\!\leq\! j_1\! <\!\cdots\! <\! j_{s_1}\!\leq m_2,\\
1\!\leq\! k_1\! <\!\cdots\! <\! k_{t_1}\!\leq \! m_3,r_1\!\leq\! r,s_1\!\leq\! s, t_1\!\leq\! t,\ r_1+s_1+t_1=i.
\end{array}\right.\!\!\!\right\}
\]
\noindent for all $i,r,s,t\geq0$.
Moreover, by convention $P_{r,s,t}^i=\emptyset$ when either $i$, $r$, $s$ or $t$ is negative.

\begin{Example}
Let $T=T_2(3,4,5)$ be the tree with three branches, the first one with three leaves, the second one with four leaves, and the third one with five leaves. 
Since $n=|V(T)|=16$ and $\nu_2(T)=6$, then $T$ has $10$ non-trivial critical ideals. 
Furthermore, by Theorem~\ref{TeoRed} and Proposition~\ref{T2mincar},
\[
I_{6+i}(T,X)\!=\!\begin{cases}
\displaystyle{\langle x_r\cdot P_{1,2,3}^{i-1},P_{2,3,4}^i, P_{0,2,3}^{i-3}\cdot Q_1,P_{1,0,3}^{i-4}\cdot Q_2, 
P_{1,2,0}^{i-5}\cdot Q_3, P_{0,0,3}^{i-6}\cdot Q_{1,2}, P_{0,2,0}^{i-7}\cdot Q_{1,3}\rangle}&\textrm{if }1\leq i\leq 7,\\
\\
\displaystyle{\langle P_{2,3,4}^8, P_{0,2,3}^5\cdot Q_1, P_{1,0,3}^4\cdot Q_2, P_{1,2,0}^3\cdot Q_3, P_{0,0,3}^2\cdot Q_{1,2}, P_{0,2,0}^1\cdot Q_{1,3}, Q_{2,3}\rangle}&\textrm{if }i=8,\\
\\
\displaystyle{\langle P_{0,0,3}^3\cdot Q_{1,2}, P_{0,2,0}^2\cdot Q_{1,3}, P_{1,0,0}^1\cdot Q_{2,3}\rangle}&\textrm{if }i=9.\\
\end{cases}
\]

Also, let $T=T_2(2,2,m)$ be the tree of depth two with three branches, the first two with $2$ leaves and the third one with $m$ leaves.
Since $n=|V(T)|=m+8$ and $\nu_2(T)=6$, then $T$ has $m+2$ non-trivial critical ideals.  
By Theorem~\ref{TeoRed} and Proposition~\ref{T2mincar},
\[
I_{6+i}(T,X)\!=\!
\begin{cases}
\displaystyle{\langle x_r\cdot P_{0,0,i-1}^{i-1}, P_{1,1,m-1}^i, P_{0,0,i-2}^{i-2}\cdot \{Q_1, Q_2\}, P_{0,0,i-3}^{i-3}\cdot Q_{1,2} \rangle} &\text{ if } 1\leq i\leq m-1,\\
\\
\displaystyle{\langle P_{1,1,m-1}^{m}, P_{0,0,m-2}^{m-2}\cdot \{Q_1, Q_2\}, P_{0,0,m-3}^{m-3}\cdot Q_{1,2}, Q_3 \rangle} &\text{ if } i=m,\\
\\
\langle P_{0,0,m-2}^{m-2}\cdot Q_{1,2}, P_{1,1,0}^1\cdot Q_3\rangle &\text{ if } i=m+1,
\end{cases}
\]
\end{Example}


\subsection{Wired d-regular trees}
A wired tree is a graph obtained from a tree by collapsing its leaves to a single vertex.
This term was introduced by Levine in~\cite{Levine}.
The critical group of a wired tree obtained from a regular tree and some variants of them have been studied in~\cite{Levine, Shen11, Toumpakary}. 

For $d\geq 3$, let $T_{d,h}$ be the rooted $d$-regular tree of depth $h$ and $\overline{T}_{d,h}$ 
the tree obtained from $T_{d,h}$ by deleting one of its principal branches.
In other words, $\overline{T}_{d,h}$ is a tree of depth $h$ in which each non-leaf vertex has $d-1$ children or successors. 
For instance, $T_{d,1}$ is a star with $d$ edges.
Now, let $WT_{d,h}$ be the wired tree obtained from $T_{d,h}$, let $W\overline{T}_{d,h}$ be the wired tree obtained from $\overline{T}_{d,h}$, 
and let $v$ be the new vertex obtained by collapsing its leaves.
Also, let $W\overline{T}'_{d,h}$ be the graph obtained from $W\overline{T}_{d,h}$ by adding an edge between $v$ and its root.

In~\cite{Levine}, Levine described completely the critical group of $W\overline{T}'_{d,h}$.
Also, in~\cite{Toumpakary}, Toumparaky studied the critical group of $WT_{d,h}$.
She calculated its rank, exponent and order.
Since $WT_{d,h}\setminus v$ is equal to $T_{d,h-1}$ and $W\overline{T}'_{d,h}\setminus v$ is equal to $\overline{T}_{d,h-1}$, then
applying~\cite[Proposition 3.7]{Corrales} we have that the critical groups of $WT_{d,h}$ and $W\overline{T}'_{d,h}$ 
can be determined as the evaluation of the critical ideals of $T_{d,h-1}$ and $\overline{T}_{d,h-1}$ on $x_i=d$ respectively.
Unfortunately, is difficult to compute explicitly the critical ideals of $T_{d,h-1}$ and $\overline{T}_{d,h-1}$ and their evaluations. 
However, we can extract some information from them.
For instance, it is not difficult to calculate the order of the critical groups of $WT_{d,h}$ and $W\overline{T}'_{d,h}$.
Before doing that we will introduce some notation.

Let $P_h(x,y)\in \mathbb{Z}[x,y]$ be defined recursively by
\[
P_h(x,y)=xP_{h-1}(x,y)-yP_{h-2}(x,y)
\]
with $P_{-1}(x,y)=1$ and $P_0(x,y)=x$.
It is not difficult to check that $P_1(x,y)=x^2-y$, $P_2(x,y)=x^3-2xy$ and $P_3(x,y)=x^4-3x^2y+y^2$.
Moreover, $P_h(x,y)$ has the following properties.

\begin{Proposition}\label{properties}
If $h\geq 0$, then
\begin{description}
\item[(1)] $P_h(x,y)=\sum_{i=0}^{\lfloor \frac{h+1}{2}\rfloor} (-1)^i\binom{h+1-i}{i} x^{h+1-2i}y^i$
\item[(2)] $P_h(a+1,a)=\sum_{i=0}^{h+1} a^i$.
\end{description}
\end{Proposition}
\begin{proof}
This follows easily using induction on $h$.
\end{proof}

Using these polynomials, we get the following expression for the determinant of the generalized Laplacian matrix $\overline{T}_{d,h}$ evaluated at $x_i=x$ for all $i$.

\begin{Proposition}\label{det1}
If $h\geq 0$ and $a=d-1\geq 2$, then
\[
\mathrm{det}(L(\overline{T}_{d,h},X)|_{x_i=x})=x\cdot \mathrm{det}(L(\overline{T}_{d,h-1},X)|_{x_i=x})^a-a\cdot \mathrm{det}(L(\overline{T}_{d,h-2},X)|_{x_i=x})^a\mathrm{det}(L(\overline{T}_{d,h-1},X)|_{x_i=x})^{a-1}.
\]
Moreover, $\mathrm{det}(L(\overline{T}_{d,h},X)|_{x_i=x})=P_h(x,a)\prod_{i=0}^{h-1} P_i(x,a)^{a^{h-i}-a^{h-i-1}}$.
\end{Proposition}
\begin{proof}
Let $r$ be the root of $\overline{T}_{d,h}$ and $\{v_1,\ldots,v_a\}$ its children.
Also, given a vertex $v$  of $\overline{T}_{d,h}$, let $\overline{T}^v_{d,h}$ be the subtree of $\overline{T}_{d,h}$ rooted in $v$.
For instance, $\overline{T}^{v_i}_{d,h}$ is isomorphic to $\overline{T}_{d,h-1}$.
Since $\overline{T}_{d,h}$ has no cycles, then using the expression for the determinant of a generalized Laplacian matrix given in~\cite[Proposition 4.1]{Corrales} 
(see also~\cite[Lemma 4.4]{Corrales} for an expression in the special case of trees, as in our case) we get that
{\small
\begin{eqnarray*}
\mathrm{det}(L(\overline{T}_{d,h},X)|_{x_i=x})&\!\!\!\!=\!\!\!\!&\sum_{\mu \in \mathcal{V}_1(\overline{T}_{d,h})} \!\!\!\!\!\!\!\!(-1)^{|\mu|}x^{[a]_h-2|\mu|}=
\sum_{rv_1,\ldots,rv_a\not\in \mu\in \mathcal{V}_1(\overline{T}_{d,h})} \!\!\!\!\!\!\!\!\!\!\!\!\!\!\!\!(-1)^{|\mu|}x^{[a]_h-2|\mu|} + \sum_{i=1}^a\sum_{rv_i\in \mu\in \mathcal{V}_1(\overline{T}_{d,h})} \!\!\!\!\!\!\!\!\!\!\!\!(-1)^{|\mu|}x^{[a]_h-2|\mu|}\\
&\!\!\!\!=\!\!\!\!& x\cdot \prod_{j=1}^a \sum_{\mu\in \mathcal{V}_1(\overline{T}^{v_j}_{d,h})} \!\!\!\!\!\!\!\!(-1)^{|\mu|}x^{[a]_{h-1}-2|\mu|}\\
&& - \sum_{i=1}^a(\prod_{u\in \mathrm{Ch}(v_i)} \sum_{\mu\in \mathcal{V}_1(\overline{T}^{u}_{d,h})} \!\!\!\!\!\!\!\!(-1)^{|\mu|}x^{[a]_{h-2}-2|\mu|}\cdot \prod_{j\neq i} \sum_{\mu\in \mathcal{V}_1(\overline{T}^{v_i}_{d,h})} \!\!\!\!\!\!\!\!(-1)^{|\mu|}x^{[a]_{h-1}-2|\mu|}) \\
&\!\!\!\!=\!\!\!\!& x\cdot \mathrm{det}(L(\overline{T}_{d,h-1},X)|_{x_i=x})^a-a\cdot \mathrm{det}(L(\overline{T}_{d,h-2},X)|_{x_i=x})^a\mathrm{det}(L(\overline{T}_{d,h-1},X)|_{x_i=x})^{a-1},
\end{eqnarray*}
}
where $\mathcal{V}_1(G)$ is the set of matchings of $G$, $[a]_h=\sum_{i=0}^h a^i$ is the number of vertices of $\overline{T}_{d,h}$, 
and $\mathrm{Ch}(v)$ is the set of children of $v$.

Finally, using induction on $h$, the properties of the polynomials $P_h(x,y)$ and the last expression for the $\mathrm{det}(L(\overline{T}_{d,h},X)|_{x_i=x})$, we get that
\[
\mathrm{det}(L(\overline{T}_{d,h},X)|_{x_i=x})=P_h(x,a)\prod_{i=0}^{h-1} P_i(x,a)^{a^{h-i}-a^{h-i-1}}.
\]
\end{proof}

\begin{Remark}
In general, using \cite[Lemma 4.4]{Corrales}, we get that if $T$ is a tree, then
\[
\mathrm{det}(L(T,X)|_{x_i=x})=x^{|V(T)|-\nu_1(T)} p(x) \text{ for some } p(x)\in \mathbb{Z}[x], 
\]
where $\nu_1(T)$ is the maximum cardinality of any matching of $T$.
\end{Remark}

Using Proposition~\ref{det1} we can easily get the order of the critical groups of $W\overline{T}'_{d,h}$ and $WT_{d,h}$.

\begin{Corollary}[\cite{Levine} p. 2]\label{order1}
If $h\geq 1$ and $a=d-1\geq 2$, then the order of $K(W\overline{T}'_{d,h})$ is equal to
\[
(1+a)^{a^{h-1}-a^{h-2}} \cdots (1+a+\cdots +a^{h-1})^{a-1} (1+a+\cdots +a^{h}).
\]
\end{Corollary}
\begin{proof}
By ~\cite[Proposition 3.7]{Corrales}  we have that the order of $K(W\overline{T}'_{d,h})$ is equal to $\mathrm{det}(L(\overline{T}_{d,h-1},X)|_{x_i=a+1})$, 
which by Propositions~\ref{det1} and~\ref{properties} is equal to $(1+a)^{a^{h-1}-a^{h-2}} \cdots (1+a+\cdots +a^{h-1})^{a-1} (1+a+\cdots +a^{h})$.
\end{proof}

\begin{Corollary}[\cite{Toumpakary} Theorem 2.10]\label{order2}
If $h\geq 1$ and $a=d-1\geq 2$, then the order of $K(WT_{d,h})$ is equal to
\[
(1+a)a^{h-1}(1+\cdots +a^{h-1})^a\prod_{i=1}^{h-2}(1+\cdots +a^{i})^{a^{h-i}-a^{h-2-i}}.
\]
\end{Corollary}
\begin{proof}
Using arguments similar to those given in Proposition~\ref{det1}, we get that
{\small
\begin{eqnarray*}
\mathrm{det}(L(T_{d,h},X)|_{x_i=x})&=&x\cdot\mathrm{det}(L(\overline{T}_{d,h-1},X)|_{x_i=x})^{1+a}-(1+a)\mathrm{det}(L(\overline{T}_{d,h-2},X)|_{x_i=x})^a\mathrm{det}(L(\overline{T}_{d,h-1},X)|_{x_i=x})^a\\
&=&\mathrm{det}(L(\overline{T}_{d,h-1},X)|_{x_i=x})^a(x\cdot\mathrm{det}(L(\overline{T}_{d,h-1},X)|_{x_i=x})-(1+a)\mathrm{det}(L(\overline{T}_{d,h-2},X)|_{x_i=x})^a).
\end{eqnarray*}
}
Using that $\mathrm{det}(L(\overline{T}_{d,h-1},X)|_{x_i=1+a})=(1+\cdots +a^{h})\prod_{i=1}^{h-1}(1+\cdots +a^{i})^{a^{h-i}-a^{h-1-i}}$, we get that
{\small
\[
(1+a)\cdot\mathrm{det}(L(\overline{T}_{d,h-1},X)|_{x_i=1+a})-(a+1)\mathrm{det}(L(\overline{T}_{d,h-2},X)|_{x_i=1+a})^a = 
(1+a)a^{h}\prod_{i=1}^{h-1}(1+\cdots +a^{i})^{a^{h-i}-a^{h-1-i}}.
\]
}
Finally, since the order of $K(WT_{d,h})$ is equal to $\mathrm{det}(L(T_{d,h-1},X)|_{x_i=1+a})$, we get the result.
\end{proof}

\begin{Remark}
Proposition~\ref{det1} can be used to compute the order of the critical group of any graph $G$ 
such that $G\setminus v$ is equal to $\overline{T}_{d,h}$ for some $v\in V(G)$ and the number of 
edges between $v$ and the vertices in $V(G)\setminus v$ is such that $\mathrm{deg}_G(u)=t$ 
for some $t\in \mathbb{N}$ and for all $u\in V(G)\setminus v$.
\end{Remark}

Calculating the rank (the number of non-trivial invariant factors) of the critical group of $W\overline{T}'_{d,h}$ and $WT_{d,h}$ is a more complicated task.
We will work with a more general class of trees.
We say that a rooted tree $(T,r)$ is an $h$-tree if the distance between the root 
and any of its leaves is equal to $h$ and any non-leaf vertex has at least two children.
First we establish a property of the $2$-matching number of this class of trees.

\begin{Lemma}\label{v2totally}
If $(T,r)$ is an $h$-tree, then
\[
\nu_2(T)=\begin{cases}
\displaystyle{\sum_{u\in \mathrm{Ch}(r)} \nu_2(T_u)}&\textrm{if }h \textrm{ is even,}\\
\\
\displaystyle{2+\sum_{u\in \mathrm{Ch}(r)}\nu_2(T_u)}&\textrm{if }h\textrm{ is odd,}
\end{cases}\]
where $T_u$ is the subtree of $T$ rooted in $u$ and $\mathrm{Ch}(v)$ is the set of children of $v$.
Moreover, the root $r$ is saturated if and only if $h$ is odd.
\end{Lemma}
\begin{proof}
We will use induction on the depth of $T$.
First, if $h=1$, then $T$ is a star and clearly $\nu_2(T)=2=2+\sum_{u\in \mathrm{Ch}(r)}\nu_2(T_u)$ 
because the $2$-matching number of the graph with only one vertex is equal to zero.

Assume that the result is true for $h$ and we will prove for $h+1$. 
By Lemma~\ref{T-v}, 
\[
\sum_{u\in \mathrm{Ch}(r)} \nu_2(T_u) \leq \nu_2(T)\leq 2+\sum_{u\in \mathrm{Ch}(r)} \nu_2(T_u).
\]
We will divide the argument into two cases: $h+1$ odd and $h+1$ even.

Assume that $h+1$ is even.
Let $\mathcal{M}$ be a maximum $2$-matching of $T$, $\mathcal{M}_r$ the edges of $\mathcal{M}$ incident to $r$ 
and $\mathcal{M}_u=\mathcal{M}\cap T_u$ for any $u\in \mathrm{Ch}(r)$.
If $|\mathcal{M}_r|=0$, then clearly $|\mathcal{M}|=\sum_{u\in \mathrm{Ch}(r)} |\mathcal{M}_u|\leq \sum_{u\in \mathrm{Ch}(r)} \nu_2(T_u)$.
Now, if $|\mathcal{M}_r|=1$, let $w$ be the vertex connected to $r$ by an edge of $\mathcal{M}$.
Since $T_u$ satisfies the induction hypothesis, has depth equal to $h$, and $w$ has degree at most one on $\mathcal{M}_w$, 
it follows that $\mathcal{M}_w$ is not maximum on $T_w$ and
\[
|\mathcal{M}|=|\mathcal{M}_r|+\sum_{u\in \mathrm{Ch}(r)} |\mathcal{M}_u|\leq 
1+\nu_2(T_w)-1+\sum_{u\neq w\in \mathrm{Ch}(r)} \nu_2(T_u)=\sum_{u\in \mathrm{Ch}(r)} \nu_2(T_u).
\] 
The case in which $\mathcal{M}$ has degree $2$ on $r$ can be treated in the same way
and therefore $\nu_2(T)=\sum_{u\in \mathrm{Ch}(r)} \nu_2(T_u)$.
Moreover, $r$ is not saturated because if $\mathcal{M}_u$ is a maximum $2$-matching of $T_u$ 
for any $u\in \mathrm{Ch}(r)$, then $\bigcup_{u\in \mathrm{Ch}(r)} \mathcal{M}_u$ is a maximum $2$-matching of $T$.

Now, assume that $h+1$ is odd.
Since $T_u$ satisfies the induction hypothesis and has depth equal to $h$, 
then for any $u\in \mathrm{Ch}(r)$ there exists a maximum $2$-matching $\mathcal{M}_u$ of $T_u$ 
such that the degree of $u$ in $\mathcal{M}_u$ is less than or equal to one ($u$ is not saturated in $T_u$).
In this case it is clear that $\mathcal{M}=\bigcup_{u\in \mathrm{Ch}(r)} \mathcal{M}_u\cup \{ru_1,ru_2\}$ 
for any $u_1\neq u_2\in \mathrm{Ch}(r)$ is a $2$-matching of $T$ and therefore $|\mathcal{M}|=2+\sum_{u\in \mathrm{Ch}(r)} \nu_2(T_u)$.
Moreover, $r$ is saturated.
\end{proof}

Now, we present a lower bound for the $2$-matching number of $T$ as a function of the number of edges of a 2-matching of $T^{\ell}$ plus
twice the matching number of the subtree of $T$ induced by the vertices that have a loop in $\mathcal{M}$.

\begin{Lemma}\label{bound}
If $(T,r)$ is an $h$-tree and $\mathcal{M}\in \mathcal{V}_2(T^{\ell})$, then
\[
|e(\mathcal{M})|+2\nu_1(T[V(\ell(\mathcal{M}))]) \leq \nu_2(T).
\]
\end{Lemma}
\begin{proof}
We will use induction on the depth of $T$.
First, assume that $h=1$, that is, $T$ is a star.
If $\mathcal{M}$ has no edges, then $\mathcal{M}=\ell(\mathcal{M})$ and $\nu_1(T[V(\ell(\mathcal{M}))])\leq 1$.
Thus the result follows because $\gamma_{\mathbb{Z}}(T)=2$.
In the other case, $\nu_1(T[V(\ell(\mathcal{M}))])=0$ and the result follows because the number of edges in $\mathcal{M}$ is at most $\gamma_{\mathbb{Z}}(T)=2$.

Assume that the result is true for $h$, and we will prove for $h+1$. 
We will divide the proof into two cases: $r\in V(\ell(\mathcal{M}))$ and $r\in V(e(\mathcal{M}))$.

First, assume that $r\in V(\ell(\mathcal{M}))$.
If $r$ is not incident to an edge of a maximum matching of $T[V(\ell(\mathcal{M}))]$, then using the induction hypothesis
\[
|e(\mathcal{M})|+2\nu_1(T[V(\ell(\mathcal{M}))])= \sum_{u\in \mathrm{Ch}(r)} |e(\mathcal{M}_u)|+2\nu_1(T[V(\ell(\mathcal{M}_u))]) \leq \sum_{u\in \mathrm{Ch}(r)}\nu_2(T_u) \leq \nu_2(T),
\]
where $\mathcal{M}_u=\mathcal{M}\cap T_u$ and $T_u$ is the subtree of $T$ rooted in $u$.
Now, assume that $r$ is incident to an edge of a maximum matching of $T[V(\ell(\mathcal{M}))]$.
Let $w\in V(T)$ be such that $rw$ is an edge of a maximum matching of $T[V(\ell(\mathcal{M}))]$.
Thus using the induction hypothesis and Lemma~\ref{v2totally}
{\small
\begin{eqnarray*}
|e(\mathcal{M})|+2\nu_1(T[V(\ell(\mathcal{M}))])&=& 2+\!\!\!\!\!\!\!\!\sum_{w\neq u\in \mathrm{Ch}(r)} \!\!\!\!\!\!\!\! 
|e(\mathcal{M}_u)|+2\nu_1(T[V(\ell(\mathcal{M}_u))])
+ \sum_{u\in \mathrm{Ch}(w)} \!\!\!\!|e(\mathcal{M}_u)|+2\nu_1(T[V(\ell(\mathcal{M}_u))])\\
&\leq& 2+ \sum_{w\neq u\in \mathrm{Ch}(r)} \nu_2(T_u) + \sum_{u\in \mathrm{Ch}(w)} \nu_2(T_u) \overset{\mathrm{Lemma}~\ref{v2totally}}{=} \nu_2(T).
\end{eqnarray*}
}

Now, asumme that $r\in V(e(\mathcal{M}))$.
If $h$ is odd, then using the induction hypothesis and Lemma~\ref{v2totally}
{\small
\begin{eqnarray*}
|e(\mathcal{M})|+2\nu_1(T[V(\ell(\mathcal{M}))])&\leq& |e(\mathcal{M}_r)|+2\nu_1(T[V(\ell(\mathcal{M}_r))])
+ \sum_{u\in \mathrm{Ch}(r)} |e(\mathcal{M}_u)|+2\nu_1(T[V(\ell(\mathcal{M}_u))])\\
&\leq& 2+\sum_{u\in \mathrm{Ch}(r)}\nu_2(T_u) \overset{\mathrm{Lemma}~\ref{v2totally}}{=} \nu_2(T),
\end{eqnarray*}
}
where $T_r=T[\{r\}\cup \mathrm{Ch}(r)]$ and $\mathcal{M}_r= \mathcal{M}_r \cap T_r$.
In a similar way, if $h$ is even, then using the induction hypothesis and Lemma~\ref{v2totally}
{\small
\begin{eqnarray*}
|e(\mathcal{M})|+2\nu_1(T[V(\ell(\mathcal{M}))])&\leq& |e(\mathcal{M}_r)|+2\nu_1(T[V(\ell(\mathcal{M}_r))])
+ \sum_{u\in \mathrm{Ch}(r)} \sum_{v\in \mathrm{Ch}(u)} |e(\mathcal{M}_v)|+2\nu_1(T[V(\ell(\mathcal{M}_v))])\\
&\leq& \nu_2(T_r)+\sum_{u\in \mathrm{Ch}(r)} \sum_{v\in \mathrm{Ch}(u)} \nu_2(T_v) \overset{\mathrm{Lemma}~\ref{v2totally}}{\leq} \nu_2(T),
\end{eqnarray*}
}
where $T_r=T[\{r\}\cup \mathrm{Ch}(r) \bigcup \cup_{u\in \mathrm{Ch}(r)} \mathrm{Ch}(u)]$ and $\mathcal{M}_r= \mathcal{M}_r \cap T_r$.
Note that $\nu_2(T_r)=2|\mathrm{Ch}(r)|$.
\end{proof}

Directly from Lemma~\ref{bound} we get the following result.

\begin{Corollary}
If $(T,r)$ is an $h$-tree, then
\[
\nu_2(T)=2\nu_1(T).
\]
\end{Corollary}
\begin{proof}
Taking $\mathcal{M}$ equal to the $2$-matching composed by a loop in each vertex of $T$,
Theorem~\ref{bound} implies that $2\nu_1(T)\leq \nu_2(T)$.
The reverse inequality is valid in general.
\end{proof}

Moreover, in this case we can  get a partial description of the critical ideals of $T$ evaluated at $x_i=x$ for all $i$.
\begin{Theorem}\label{divx}
If $(T,r)$ is an $h$-tree, $1\leq i\leq |V(T)|-\gamma_{\mathbb{Z}}(T)$ and
\[
I_{\gamma_{\mathbb{Z}}(T)+i}(T)|_{x_i=x}=\left\langle p_1(x),\ldots, p_s(x)\right\rangle,
\] 
then $p_j(x)=x^iq_j(x)$ for some $q_j(x)\in\mathbb{Z}[x]$.
Moreover, if $|V(T)|\geq 4$, then $I_{\gamma_{\mathbb{Z}}(T)+1}(T)|_{x_i=x}=\left\langle x\right\rangle$.
\end{Theorem}
\begin{proof}
By Corollary~\ref{description1}, $I_j(T, X)=\langle d(\mathcal{M}, X) \, | \, \mathcal{M}\in\mathcal{V}_2(T^\ell, j)\rangle$.
Thus
\[
I_j(T, X)|_{x_i=x}=\langle d(\mathcal{M}, X)|_{x_i=x} \, | \, \mathcal{M}\in\mathcal{V}_2(T^\ell, j)\rangle.
\]
Now, let $\mathcal{M}\in\mathcal{V}_2(T^\ell,j)$, $\nu_1=\nu_1(T[V(\ell(\mathcal{M}))])$ and $p(x)=d(\mathcal{M}, X)|_{x_i=x}$.
By Lemma~\ref{Lemmarel} and~\cite[Lemma 4.4]{Corrales}
\[
p(x)=x^{|\ell(\mathcal{M})|-2\nu_1}q(x) \text{ for some }q(x)\in\mathbb{Z}[x].
\]
Finally, by Lemma~\ref{bound} and Theorem~\ref{TeoTreeGamma},  
$|\ell(\mathcal{M})|-2\nu_1\geq |\mathcal{M}|-\nu_2(T)=|\mathcal{M}|-\gamma_{\mathbb{Z}}(T)$
and we get that $p_j(x)=x^iq_j(x)$ for some $q_j(x)\in\mathbb{Z}[x]$.
Moreover, is not difficult to see that any $h$-tree $T$ with $|V(T)|\geq 4$ 
has a $2$-matching of $T^\ell$ of size $\nu_2(T)+1$ with a leaf as a loop and no other loop. 
\end{proof}

Using these results we can get the rank of the critical group of the following family of graphs.

\begin{Corollary}\label{rank}
Let $G$ be a graph and $v$ a vertex of $G$.
If $G\setminus v$ is an $h$-tree and $\mathrm{deg}_G(u)=\mathrm{deg}_G(w)$ for all $u\neq w \in V(G)\setminus v$, 
then the rank of the critical group of $G$ is equal to
\[
|V(G\setminus v)|-\gamma_{\mathbb{Z}}(G\setminus v).
\]
Moreover, its first invariant factor is equal to $\mathrm{deg}_G(u)$.
\end{Corollary}
\begin{proof}
This follows directly from Theorem~\ref{divx}.
\end{proof}

The next example shows how Theorem~\ref{divx} works when $T$ is equal to $\overline{T}_{3,3}$.
\begin{Example}\label{h3d3}
If $h=3$ and $d=3$, then is not difficult to check that $\nu_2(\overline{T}_{d,h})=10$ and 
\begin{eqnarray*}
I_{11}(\overline{T}_{3,3},X)|_{x_i=x}&=&\left\langle x\right\rangle,\\
I_{12}(\overline{T}_{3,3},X)|_{x_i=x}&=&\left\langle x^2\right\rangle,\\
I_{13}(\overline{T}_{3,3},X)|_{x_i=x}&=&\left\langle 2x^3, x^5\right\rangle,\\
I_{14}(\overline{T}_{3,3},X)|_{x_i=x}&=&\left\langle 4x^4(x^2-2), x^4(x^2-2)(x^2+2)\right\rangle,\\
I_{15}(\overline{T}_{3,3},X)|_{x_i=x}&=&\left\langle x^5(x^2-2)^2(x^2-4)(x^4-6x^2+4)\right\rangle.
\end{eqnarray*}
Note that in general the critical ideals evaluated at $x_i=x$ are not principal and $x^j$ divides the generators of $I_{10+j}(\overline{T}_{3,3},X)|_{x_i=x}$.

On the other hand, 
\[
I_{11}(\overline{T}_{3,3},X)=\left\langle x_8,x_9,x_{10},x_{11},x_{12},x_{13},x_{14},x_{15},x_1x_2x_3-x_2-x_3\right\rangle,
\]
which shows that evaluating the critical ideals at $x_i=x$ greatly simplifies the descriptions of the ideals.
\end{Example}

Is not difficult to find a tree such that the generators of its critical ideals are divided by $x$, as the next example shows.
\begin{Example}
Consider the tree $T$ given in Figure~\ref{ejemplo}.
\begin{figure}[h!]
\begin{center}
\begin{tikzpicture}
\draw {
	(234-90:1) node[draw, circle, fill=gray, inner sep=0pt, minimum width=4pt] (v1) {}
	(162-90:1) node[draw, circle, fill=gray, inner sep=0pt, minimum width=4pt] (v2) {}
	(90-90:1) node[draw, circle, fill=gray, inner sep=0pt, minimum width=4pt] (v3) {}
	(18-90:1) node[draw, circle, fill=gray, inner sep=0pt, minimum width=4pt] (v4) {}
	(-54-90:1) node[draw, circle, fill=gray, inner sep=0pt, minimum width=4pt] (v5) {}
	(2,0) node[draw, circle, fill=gray, inner sep=0pt, minimum width=4pt] (v6) {}
	
	(v1) to (v2) 
	(v2) to (v3) 
	(v3) to (v4) 
	(v4) to (v5) 
	(v3) to (v6) 
	
	(v1)+(0,0.2) node {$v_1$}
	(v2)+(0,0.2) node {$v_2$}
	(v3)+(0.2,0.2) node {$v_3$}
	(v4)+(0,-0.2) node {$v_4$}
	(v5)+(0,-0.2) node {$v_5$}
	(v6)+(0.2,0.2) node {$v_6$}
};
\end{tikzpicture}
\caption{A tree with six vertices.} 
\label{ejemplo}
\end{center}
\end{figure}
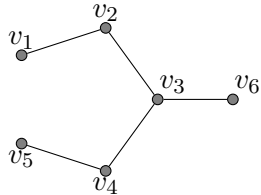

It is not difficult to check that $\nu_2(T)=4$ and 
\[
I_{5}(T,X)=\left\langle x_1x_2-1, x_4x_5-1,x_6\right\rangle.
\]
Thus $I_{5}(T,X)|_{x_i=x}=\left\langle 1\right\rangle$.
Moreover, it also can be checked that $I_{6}(T,X)|_{x_i=x}=\left\langle (x^2-1)(x^4-4x^2+1)\right\rangle$.
\end{Example}

By Corollary~\ref{rank}, in order to get the rank of $K(WT_{d,h})$ and $K(W\overline{T}'_{d,h})$,
we only need to compute their matching numbers.

\begin{Corollary}\label{mintnr}
If $h\geq 1$ and $a=d-1\geq 2$, then
\[
\nu_2(\overline{T}_{d,h})=\begin{cases}
\displaystyle{2\frac{a^{h+1}-a}{a^2-1}}&\textrm{if }h \textrm{ is even,}\\
\\
\displaystyle{2\frac{a^{h+1}-1}{a^2-1}}&\textrm{if }h\textrm{ is odd.}
\end{cases}\]
\end{Corollary}
\begin{proof}
This follows by Lemma~\ref{v2totally}.
\end{proof}

In a similar way.
\begin{Corollary}
If $h\geq 1$ and $a=d-1\geq 2$, then
\[
\nu_2(T_{d,h})=2\frac{a^h-1}{a-1}=2(1+a+\cdots+a^{h-1}).
\]
\end{Corollary}
\begin{proof}
Since $\nu_2(T_{d,h})=\nu_2(\overline{T}_{d,h})+\nu_2(\overline{T}_{d,h-1})$, the result follows from Corollary~\ref{mintnr}.
\end{proof}

Thus we get the rank of $K(WT_{d,h})$ and $K(W\overline{T}'_{d,h})$. 

\begin{Corollary}
If $h\geq 1$ and $a=d-1$, then the critical group of $W\overline{T}'_{d,h}$ has rank 
\[
\sum_{i=0}^{h-1}(-1)^ia^{h-1-i}.
\] 
Furthermore, its first non-trivial invariant factor is equal to $d$. 
\end{Corollary}
\begin{proof}
Using Corollary~\ref{mintnr} it is not difficult to check that
\[
\nu_2(\overline{T}_{d,h-1})=2\sum_{i=0}^{\lfloor \frac{h-2}{2}\rfloor}a^{h-2-2i}
\]
and therefore the rank of the critical group of $W\overline{T}'_{d,h}$ is equal to
\[
|V(\overline{T}_{d,h-1})|-\nu_2(\overline{T}_{d,h-1})=\sum_{i=0}^{h-1}(-1)^ia^{h-1-i}.
\]  
\end{proof}

\begin{Corollary}\label{CGTdh}
If $h\geq 1$ and $a=d-1$, then the critical group of $WT_{d,h}$ has rank $a^{h-1}$ and
its first non-trivial invariant factor is equal to $d$. 
\end{Corollary}
\begin{proof}
First, since $\nu_2(T_{d,h-1})=2(1+a+\cdots+a^{h-2})$ and
\[
|V(T_{d,h-1})|=1+d+d(d-1)+\cdots+d(d-1)^{h-2}=2+2a+\cdots+2a^{h-2}+a^{h-1},
\]
then the rank the critical group of $WT_{d,h}$ is equal to $|V(T_{d,h-1})|-\nu_2(T_{d,h-1})=a^{h-1}$.
\end{proof}

Finally, we present the critical ideals of $\overline{T}_{d,h}$ for $h=1$ and $h=2$.
\begin{Corollary}
If $a=d-1\geq 2$, then
\[
I_{j}(\overline{T}_{d,1},X)|_{x_i=x}=
\begin{cases}
\left\langle 1\right\rangle & \text{ if } 1\leq j \leq 2,\\
\left\langle x^{j-2}\right\rangle & \text{ if } 3\leq j \leq a,\\
\left\langle x^{a-1}(x^2-a)\right\rangle & \text{ if } j=a+1,
\end{cases}
\]
and
\[
I_{j}(\overline{T}_{d,2},X)|_{x_i=x}=
\begin{cases}
\left\langle 1\right\rangle & \text{ if } 1\leq j \leq 2a,\\
\left\langle x^{j-2a}\right\rangle & \text{ if } 2a+1\leq j \leq a^2+2,\\
\left\langle x^{j-2a} (x^2-a)^{j-a^2-2}\right\rangle & \text{ if } a^2+3\leq j \leq a^2+a,\\
\left\langle x^{a^2-a}(x^2-a)^{a-1}(x^3-2ax)\right\rangle & \text{ if }  j=a^2+a+1.
\end{cases}
\]
\end{Corollary}
As Example~\ref{h3d3} shows, the case for $h\geq 3$ is more complicated than these two previous cases.


\subsection{Arithmetical trees}
An \emph{arithmetical graph} is a triplet $(G,{\bf d},{\bf r})$ given by a graph $G$ and ${\bf d},{\bf r}\in\mathbb{Z}_+^{|V(G)|}$ 
such that $({\rm Diag}({\bf d})-A){\bf r}=0$, where $A$ is the adjacency matrix of $G$.
Any graph $G$ belongs to an arithmetical graph in a natural way, just taking ${\bf d}$ as its degree vector and ${\bf r}=(1,\ldots,1)^t$.
The matrix $M={\rm Diag}({\bf d})-A$ arises in algebraic geometry as an intersection 
matrix of degenerating curves, see~\cite{Lorenzini89, Lorenzini90} and the references contained there for more details.  

Given an arithmetical graph $(G,{\bf d},{\bf r})$, we define its critical group $K(G,{\bf d},{\bf r})$ (also called the group of components) 
as the torsion part of $\mathbb{Z}^{|V(G)|}/\textrm{Im}(M)$.
In~\cite{Lorenzini89}, Lorenzini proved that the $\mathbb{Z}$-rank of $K(G,{\bf d},{\bf r})$ is equal to $n-1$. 
Furthermore, if the Smith Normal Form of $M$ is $\textrm{diag}(f_1,\ldots,f_{n-1},0)$, then $K(G,{\bf d},{\bf r})=\mathbb{Z}_{f_1}\oplus\cdots\oplus\mathbb{Z}_{f_{n-1}}$. 
Since $M=L(G,{\bf d})$ and $\prod_{i=1}^j f_i$ is the greatest common divisor of the $j$-minors of $M$ for each $1\leq j \leq n-1$, 
it follows that $\langle \prod_{i=1}^j f_i \rangle$ is the generator of the $j$-critical ideal of $G$ evaluated at ${\bf d}$.

Thus, the invariant factors of $K(G,{\bf d},{\bf r})$ can be found as follows:
First, find a set of generators of the critical ideals of $G$. 
After that, we evaluate them at ${\bf d}$ and finally compute the greatest common divisor.
For instance, consider the family of arithmetical graphs associated to the reduction of elliptic curves of Kodaira type $I_n^*$. 
For any $m\in\mathbb{N}$, let $C_{5,m}$ be the tree obtained by identifying the center of a star with two leaves with each leaf of the path $P_{m+1}$, see Figure~\ref{C5n}. 
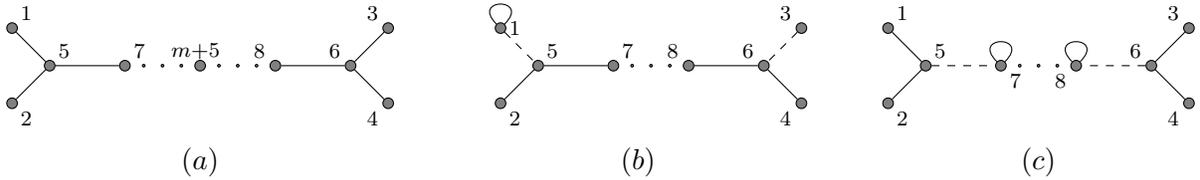
\begin{figure}[ht]
\begin{center}
\begin{tabular}{c@{\extracolsep{10mm}}c@{\extracolsep{10mm}}c}
\begin{tikzpicture}
\draw (0,0) {
	(-2,0.5) node[draw, circle, fill=gray, inner sep=0pt, minimum width=4pt] (v1) {}
	(-2,-0.5) node[draw, circle, fill=gray, inner sep=0pt, minimum width=4pt] (v2) {}
	(-1.5,0) node[draw, circle, fill=gray, inner sep=0pt, minimum width=4pt] (v5) {}
	(-0.5,0) node[draw, circle, fill=gray, inner sep=0pt, minimum width=4pt] (v7) {}
	(-0.25,0) node[draw, circle, fill=gray, inner sep=0pt, minimum width=1pt] () {}
	(0,0) node[draw, circle, fill=gray, inner sep=0pt, minimum width=1pt] () {}
	(0.25,0) node[draw, circle, fill=gray, inner sep=0pt, minimum width=1pt] () {}
	(0.5,0) node[draw, circle, fill=gray, inner sep=0pt, minimum width=4pt] (vm5) {}
	(0.75,0) node[draw, circle, fill=gray, inner sep=0pt, minimum width=1pt] () {}
	(1,0) node[draw, circle, fill=gray, inner sep=0pt, minimum width=1pt] () {}
	(1.25,0) node[draw, circle, fill=gray, inner sep=0pt, minimum width=1pt] () {}
	(1.5,0) node[draw, circle, fill=gray, inner sep=0pt, minimum width=4pt] (v8) {}
	(2.5,0) node[draw, circle, fill=gray, inner sep=0pt, minimum width=4pt] (v6) {}
	(3,-0.5) node[draw, circle, fill=gray, inner sep=0pt, minimum width=4pt] (v4) {}
	(3,0.5) node[draw, circle, fill=gray, inner sep=0pt, minimum width=4pt] (v3) {}
	
	(v1) to (v5) 
	(v2) to (v5) 
	(v5) to (v7) 
	(v3) to (v6) 
	(v4) to (v6) 
	(v6) to (v8)
	
	(v1)+(0.2,0.2) node {$_1$}
	(v2)+(0.2,-0.2) node {$_2$}
	(v3)+(-0.2,0.2) node {$_3$}
	(v4)+(-0.2,-0.2) node {$_4$}
	(v5)+(0.2,0.2) node {$_5$}
	(v6)+(-0.2,0.2) node {$_6$}
	(v7)+(0.2,0.2) node {$_7$}
	(v8)+(-0.2,0.2) node {$_8$}
	(vm5)+(-0.05,0.2) node {$_{m+5}$}
		};
\end{tikzpicture}
&
\begin{tikzpicture}[every loop/.style={}]
\draw (0,0) {
	(-2.,0.5) node[draw, circle, fill=gray, inner sep=0pt, minimum width=4pt] (v1) {}
	(-2,-0.5) node[draw, circle, fill=gray, inner sep=0pt, minimum width=4pt] (v2) {}
	(-1.5,0) node[draw, circle, fill=gray, inner sep=0pt, minimum width=4pt] (v5) {}
	(-0.5,0) node[draw, circle, fill=gray, inner sep=0pt, minimum width=4pt] (v7) {}
	(-0.25,0) node[draw, circle, fill=gray, inner sep=0pt, minimum width=1pt] () {}
	(0,0) node[draw, circle, fill=gray, inner sep=0pt, minimum width=1pt] () {}
	(0.25,0) node[draw, circle, fill=gray, inner sep=0pt, minimum width=1pt] () {}
	(0.5,0) node[draw, circle, fill=gray, inner sep=0pt, minimum width=4pt] (v8) {}
	(1.5,0) node[draw, circle, fill=gray, inner sep=0pt, minimum width=4pt] (v6) {}
	(2,-0.5) node[draw, circle, fill=gray, inner sep=0pt, minimum width=4pt] (v4) {}
	(2,0.5) node[draw, circle, fill=gray, inner sep=0pt, minimum width=4pt] (v3) {}
	
	(v1) edge[dashed] (v5)
	(v2) to (v5)
	(v5) to (v7) 
	(v3) edge[dashed] (v6)
	(v4) to (v6)
	(v6) to (v8)
	(v1) edge[loop] (v1)
	
	(v1)+(0.2,0.0) node {$_1$}
	(v2)+(0.2,-0.2) node {$_2$}
	(v3)+(-0.2,0.2) node {$_3$}
	(v4)+(-0.2,-0.2) node {$_4$}
	(v5)+(0.2,0.2) node {$_5$}
	(v6)+(-0.2,0.2) node {$_6$}
	(v7)+(0.2,0.2) node {$_7$}
	(v8)+(-0.2,0.2) node {$_8$}
	};
\end{tikzpicture}
&
\begin{tikzpicture}[every loop/.style={}]
\draw (0,0) {
	(-2,0.5) node[draw, circle, fill=gray, inner sep=0pt, minimum width=4pt] (v1) {}
	(-2,-0.5) node[draw, circle, fill=gray, inner sep=0pt, minimum width=4pt] (v2) {}
	(-1.5,0) node[draw, circle, fill=gray, inner sep=0pt, minimum width=4pt] (v5) {}
	(-0.5,0) node[draw, circle, fill=gray, inner sep=0pt, minimum width=4pt] (v7) {}
	(-0.25,0) node[draw, circle, fill=gray, inner sep=0pt, minimum width=1pt] () {}
	(0,0) node[draw, circle, fill=gray, inner sep=0pt, minimum width=1pt] () {}
	(0.25,0) node[draw, circle, fill=gray, inner sep=0pt, minimum width=1pt] () {}
	(0.5,0) node[draw, circle, fill=gray, inner sep=0pt, minimum width=4pt] (v8) {}
	(1.5,0) node[draw, circle, fill=gray, inner sep=0pt, minimum width=4pt] (v6) {}
	(2,-0.5) node[draw, circle, fill=gray, inner sep=0pt, minimum width=4pt] (v4) {}
	(2,0.5) node[draw, circle, fill=gray, inner sep=0pt, minimum width=4pt] (v3) {}
	
	(v1) to (v5) 
	(v2) to (v5) 
	(v5) edge[dashed] (v7) 
	(v3) to (v6) 
	(v4) to (v6) 
	(v6) edge[dashed] (v8)
	(v7) edge[loop] (v7)
	(v8) edge[loop] (v8)
	
	(v1)+(0.2,0.2) node {$_1$}
	(v2)+(0.2,-0.2) node {$_2$}
	(v3)+(-0.2,0.2) node {$_3$}
	(v4)+(-0.2,-0.2) node {$_4$}
	(v5)+(0.2,0.2) node {$_5$}
	(v6)+(-0.2,0.2) node {$_6$}
	(v7)+(0.2,-0.2) node {$_7$}
	(v8)+(-0.2,-0.2) node {$_8$}
	};
\end{tikzpicture}\\
$(a)$ & $(b)$ & $(c)$
\end{tabular}
\caption{The tree $C_{5,m}$ and the two types of $2$-matchings of size $m+3$.}
\label{C5n}
\end{center}
\end{figure}

Now, we will describe the critical ideals of $C_{5,m}$.
First, since $V(C_{5,m})\setminus \{v_1,v_3\}$ induces a path isomorphic to $P_{m+3}$, it follows that $\nu_2(C_{5,m})\geq m+2$.
Moreover, is not difficult to check that $\nu_2(C_{5,m})=m+2$.
Thus, by Theorem~\ref{TeoTreeGamma}, $\gamma(C_{5,m})=m+2$ and $C_{5,m}$ has only $3$ non-trivial critical ideals.
The $m+5$-critical ideal is the determinant of the generalized Laplacian matrix.
For simplicity, we will assume that $m\geq 5$.
By Proposition~\ref{paths} we get that
\[
I_{m+4}(C_{5,m},X)=\langle x_1x_3,x_1x_4,x_2x_3,x_2x_4,\mathcal{P}_{1,2}\mathcal{P}_{7,8}-x_1x_2\mathcal{P}_{9,8},\mathcal{P}_{3,4}\mathcal{P}_{7,8}-x_3x_4\mathcal{P}_{7,10}\rangle,
\]
where $\mathcal{P}_{i,j}=\det(P(v_i,v_j))$ and $P(v_i,v_j)$ is the unique path in $C_{5,m}$ that join the vertices $v_i$ and $v_j$.
Note that $\mathrm{det}(C_{5,m}\setminus P(v_3,v_4), X)=\mathcal{P}_{1,2}\mathcal{P}_{7,8}-x_1x_2\mathcal{P}_{9,8}$ 
and similarly in the case of $\mathrm{det}(C_{5,m}\setminus P(v_1,v_2),X)$.
Finally, in Figure \ref{C5n} are sketched the two types of minimal $2$-matchings of $C_{5,m}$ of size $m+3$.
Thus
\[
I_{m+3}(C_{5,m},X)=\langle x_1,x_2,x_3,x_4,\mathcal{P}_{7,8}\rangle.
\]

Now, taking ${\bf d}_{5,m}=(2,\ldots,2)$ and ${\bf r}_{5,m}=(1,1,1,1,2,\ldots,2)^t$ we get that $(C_{5,m},{\bf d}_{5,m},{\bf r}_{5,m})$ is an arithmetical graph
since $\gamma(C_{5,m})=m+2$, $f_{i}=1$ for all $1\leq i\leq m+2$.
On the other hand, using \cite[Corollary 4.5]{Corrales} we get that the polynomial $\mathcal{P}_{i,j}$ evaluated at $d=(2,\ldots, 2)$ 
is odd if and only if the path $P(v_i,v_j)$ has an even number of vertices and $\mathcal{P}_{1,2}$ and $\mathcal{P}_{3,4}$ evaluated at $(2,2,2)$ are equal to $4$.
Thus, $f_{m+3}=1$ when $m$ is odd, $f_{m+3}=2$ when $m$ is even.
Finally, since $f_{m+3}f_{m+4}=I_{m+4}(C_{5,m},(2,\ldots, 2))=4$, then
\[
K\big((C_{5,m},{\bf d}_{5,m},{\bf r}_{5,m})\big)=\begin{cases}
\mathbb{Z}_{2}^2&\textrm{if }m\textrm{ is even,}\\
\mathbb{Z}_{4}&\textrm{if }m\textrm{ is odd.}
\end{cases}
\]



\begin{thebibliography}{15}

\bibitem{Adams}{W.W. Adams and P. Loustaunau, An Introduction to Gr\"obner Bases, Providence, RI, Amer. Math. Soc., 1994.}

\bibitem{Alfaro}{C.A. Alfaro and C.E. Valencia, Graphs with two trivial critical ideals, Discrete Appl. Math. 167 (2014) 33--44.}

\bibitem{AV}{C.A. Alfaro and C.E. Valencia, Small clique number graphs with three trivial critical ideals, manuscript, ArXiv:1311.5927.}

\bibitem{AVE}{C.A. Alfaro and C.E. Valencia, Graphs with few trivial critical ideals, Electronics Notes in Discrete Mathematics 50 (2015) 391--396.}

\bibitem{twin}{C.A. Alfaro, H. Corrales, C.E. Valencia, Critical ideals of signed graphs with twin vertices, Adv. in Appl. Math. 86 (2017), 99Ð131. ArXiv:1504.06257}

\bibitem{Biggs99}{N.L. Biggs, Chip-firing and the critical group of a graph, J. Algebraic Combin. 9 (1999), no. 1, 25--45.}

\bibitem{Corrales}{H. Corrales and C.E. Valencia, On the critical ideals of graphs, Linear Algebra Appl. 439 (2013), no. 12, 3870--3892}

\bibitem{Levine}{L. Levine, The sandpile group of a tree, European J. Combin. 30 (2009), no. 4, 1026--1035.}

\bibitem{Lorenzini89}{D.J. Lorenzini, Arithmetical graphs, Math. Ann. 285 (1989), no. 3, 481--501.}

\bibitem{Lorenzini90}{D.J. Lorenzini, Dual graphs of degenerating curves, Math. Ann. 287 (1990), no. 1, 135--150.}

\bibitem{Shen11}{X. Shen and Y. Hou, The sandpile group of a bilateral regular tree. Australas. J. Combin. 51 (2011), 61--75. }

\bibitem{Toumpakary}{E. Toumpakari, On the sandpile group of regular trees, European J. Combin. 28 (2007), no. 3, 822--842. }

\end{thebibliography}
\end{document}